%% file: energyBalancing_generic.tex
\documentclass[a4paper,12pt]{article}
\usepackage[sfdefault=cmbr,OMLmathsans]{isomath}  
\usepackage{upgreek}  

\usepackage{tensor}  

\pdfoutput=1

\usepackage{amsmath}
\usepackage{amssymb}
\usepackage{amsthm}

\usepackage{subfigure}
\usepackage[hyphens]{url}

\theoremstyle{plain}
\newtheorem{theorem}{Theorem}[section]

\newtheorem{lemma}[theorem]{Lemma}
\newtheorem{proposition}[theorem]{Proposition}

\theoremstyle{definition}
\newtheorem{definition}[theorem]{Definition}

\theoremstyle{remark}
\newtheorem{remark}[theorem]{Remark}

\usepackage[colorinlistoftodos]{todonotes}
\usepackage{pst-all}
\usepackage{wrapfig}

\usepackage{array}

\newcommand{\PsTricks}{./PsTricks}

\newcommand{\vek}[1]{\mathchoice{\displaystyle\boldsymbol{#1}}
{\textstyle\boldsymbol{#1}}{\scriptstyle\boldsymbol{#1}}
{\scriptscriptstyle\boldsymbol{#1}}}
\renewcommand{\vec}[1]{{\vek #1}}
\newcommand{\ten}[1]{\tensor[]{ #1}{}}
\def\norm#1{  \left\| #1 \right\|}
\def\grad{\nabla}

\def\scalar#1#2{ \left< #1 , #2 \right>}

\begin{document}



\title{On meeting Energy Balance Errors in Cosimulations}

\author{Thilo Moshagen$^{\rm a}$ $^{\ast}$\thanks{$^\ast$Corresponding author. Email: t.moshagen@tu-bs.de
\vspace{6pt}} \\ $^{a}${\em{Institut f\"{u}r Wiss. Rechnen,  TU Braunschweig, Germany}};}

\maketitle
\input{abstractEnergy}


\input{introductionEnergy}

While under all circumstances it is desirable to guess the influence
information at the subsystem boundaries from past data as well as
possible, balance correction techniques bear the profound problem that
for establishing conservation \emph{ a posteriori} over the whole time
an instantanious error in the exchanged data has to be
accepted. More precisely, balance correction means making an error in 
the exchanged signal $u$,      for lowering the accumulated error in the amount of  that quantity  and thus lowering the error in states $x$ that would be caused by this 
now persistent error in amount.
 The purposefully commited error which the refeed of past errors actually is of 
 course bears an  error in the derivatives of the exchanged data in it, which might cause dynamics in the receiving system and its neighbours, especially if subsystems act on quicker time scales. 
  Our past work \cite{ScharffMoshagen2017}  lowered the error in the
derivatives by construction and use of suitable functions for smoothing
during switching and adding of correction terms.\\
\\

Although in  \cite{ArnoldGuenther2001} and \cite{Moshagen2017} the convergence of explicite cosimulation methods for ODE and index one PDE was proven, and
the result in  \cite{Moshagen2017} was easily extended to balance correction methods, which thus are proven justified --  it can be easily shown that those methods are not stable, neither with balance correction \cite{Moshagen2017}. This is the most severe restriction to explicite simulator coupling. 

From seeing instability as a rise in systems energy as mentioned in \cite[Sec.4.2.2]{Moshagen2017} and \cite[Sec.3.2]{ScharffMoshagen2017} and considering the power that is given by variables at the subsystems interfaces, in this contribution a stable explicit coupling scheme is derived.


\subsection{Some preliminaries and notation}

Always, a time interval or quantities belonging to one are indexed with the index of its right boundary: $\Delta t_i := [t_{i-1},t_i)$. Indexing of times begins with 0.
Big letters refer to time exchange steps, e.g  $\Delta T_k := [T_{k-1},T_k)$ is an interval between exchange times, whereas above interval might denote a subsystems step. Consistently, $k =1,...,N $, and above $i$ ranges from $1$ to $n$ too.

Let $ \operatorname{Ext}(\vec u)_{ji}^k(t)$ denote the value
of the input variable as it is assumed to be by $S_j$, 
some extrapolation of $\vec u$, 
and    $\overline{  \operatorname{Ext}(\vec u)}_{ji}^k(t)$  be the flow as it is used
for calculation constructed on $\left[ T_{k-1}, T_k\right) $ -- it may be different from $ \operatorname{Ext}(\vec u)_{ji}^k(t)$, for example a smooth combination of  $
  \operatorname{Ext}(\vec u)_{ji}^{k-1}(t)$ and  $\operatorname{Ext}(\vec u)_{ji}^k(t)$. 
The error $\Delta E_{\vek u,k} = 
\int_{T_{k-1}}^{T_k}\vek u_{ji} - \overline{  \operatorname{Ext}(\vec u)}_{ji}^k(t) dt$ is 
the balance error if $\vek u_{ji}$ is a flux of a conserved quantity, but it is defined for arbitrary quantities. 
The expression $(\vek u)_i$ denotes the $i$-th component of the vector $\vek u$, similarly,  $(\vek u)_\mathcal I$ denotes all component of the vector $\vek u$ whose indices are in the set $\mathcal I$.

\subsection{Explicit Simulator Coupling}
In industrial environment, explicit simulator coupling cannot always be avoided. There have been efforts to standardize model interfaces in order to enable coupling into one monolitic system and solve them using one ODE/DAE solver. Such an effort is the FMI standard \cite{FMI}. But those efforts so far have not led to replacement of simulator coupling. This is, of course, due to the numerous given legacy codes, due to the fact that parallelization is also advantageous, especially if subsystems softwares are equipped with solvers that are customized to the problem, among other reasons. For example, including the residual of an FEM equation system into the monolithic global system is not sensible due to the amount of data that has to be passed to the monolitic solver, and the build-in solver usually is highly adapted to the problems needs. The solver for the global system cannot be optimal for all subsystems. \\
So, simulator coupling remains and will remain a field of work. 
\subsubsection{Iterative Methods}
One can repeat the calculation for timestep $k$ using  the inputs $ \vek u_{ji}(t)$ determined from the subsystems states numerical solution $\vek x_\Delta(t)$ just calculated - doing so all at the same time would correspond to Jacobi iteration or \emph{waveform iteration},  doing so one after the other, preferably in the order of the dependency of the systems, would correspond to the Gauss-Seidel-Scheme \cite{Busch2012}.  These schemes usually converge \cite{ArnoldGuenther2001}. One then has to program an external iteration procedure including a convergence criteria, and information exchange now concerns data that is time-dependent everywhere in $T_{i-1}, T_i$. Implementing this requires skills that are frequently not sufficiently available in the environments in question, and sometimes it is computationally cheaper to lower the exchange time step size instead of iterating.
\subsubsection{Making Inputs consistent}
The coupling equations \eqref{coupling1} - \eqref{coupling2} together with the algebraic equations \eqref{splitAlgebraic1} and \eqref{splitLast} of the  subsystems  form a global system of algebraic equations. If the graph of the information flow through the subsystems has no loops and one applies an according Gauss-Seidel scheme, then the exchanged data $ $ is consistent in the sense that the equations \eqref{coupling1}, \eqref{coupling2},  \eqref{splitAlgebraic1} and \eqref{splitLast} are fulfilled - in the general case they are not if one just solves \eqref{coupling1}, \eqref{coupling2} with respect to the input. \\
To avoid those errors, it has been suggested to solve the system  \eqref{coupling1}, \eqref{coupling2},  \eqref{splitAlgebraic1} and \eqref{splitLast} at exchange times. While this is obviously desirable for iterative and explicite schemes, it requires that subsystems solvers provide an  the residual of their algebraic  equations to callers. As mentioned above repeatedly, this need not be the case. Further, this again requires considerable programming effort, including a nonlinear solver call and convergence criteria.
\subsubsection{Trying to tackle some drawbacks in explicit coupling }
Finally, with the reasons given above,  frequently the setting allows only for explicit simulator coupling. The drawbacks of using piecewise explicit extrapolation are: 
\begin{itemize}
\item discontinuity, 
\item disbalance in amounts of conserved inputs but also in quantities depending on them, for example energy. Balance correction as given by Table \ref{tab:cosimSchemeBC}, first suggested in \cite{Scharff12}, provides relief to disbalances in conserved inputs to a degree that it makes simulations possible that would be useless without. \\
Balance correction methods were applied to  nonconserved quantities  in \cite{ScharffMoshagen2017}, as even such quantities have some conservation properties in space and time due to their continuity, and examination will go on here.
\item  high derivatives in signals which are induced by the method, especially the balance correction recontributions, which are added as product of the missed amount and a hat function with integral 1,  can cause high derivatives.
\end{itemize}

In 
 \cite{ScharffMoshagen2017} the issues of balance correction and smoothing of values as well as smoothing of derivatives were considered together, showing ways to minimize unwanted behavior and side effects of balance correction and smoothing.  This was mainly done by extending the recontribution of the correction over several time intervals. But any time delayed correction, as all balance corection methods are,  cannot establish balance in quantities depending on inputs as energy. Translating energy gain into mathematical terminology, we state that nothing can make an explicit method stable.

\subsection{Aim of this work}

The state of the art offers relief for many problems that arise when one extrapolates data during cosimulation, also to errors in balances as consequence of extrapolation errors. But finally, cosimulation is an explicit method and is not  stable (\cite{ScharffMoshagen2017}, \cite{Moshagen2017}), even when balance correction is applied. We see the instability as a rise in energy, which usually is a norm or half-norm for the system, that stems from errors in power acting on subsystems interfaces due to extrapolation errors in the factors of that power.  Even if  balance correction is applied on the factor,   the time delay with which the refeed is applied in general has side effects, as the exchanged quantity then arrives at a time when the systems  states already changed. For example  the force $f$ is a factor of the mechanical energy. If a balance correction method is applied to the impulse $p = \int f \,dt$, it is not established that the energy balance is fulfilled. This is described in Section \ref{classification} and an example is given in Section \ref{exampleAugmentationOfEnergy}.
In  Section \ref{enforcingEnergyBalance} an  exchange scheme is presented that balances exchanged energy by first choosing the power as exchanged variable and then making subsystems agree on which amount of energy should be exchanged, making the algorithm stable.

\section{Exchanging Factors of conserved Quantities}

\subsection{Convergence of cosimulation schemes}
In our preceding work   \cite{Moshagen2017} convergence rates for the standard cosimulation scheme
\begin{table}[h!tb]\label{tab:generalScheme}
\begin{center}
\begin{tabular}{p{0.27\textwidth} | p{0.27\textwidth}} 
\multicolumn{2}{ c }{General scheme}\tabularnewline
\center{$S_1$} & \center{$S_2$}\tabularnewline
\multicolumn{2}{ c }{System States}\tabularnewline
\center{$\vec x_1  $} 		&\center{$\vec x_2 $} \tabularnewline
\multicolumn{2}{ c }{Outputs}\tabularnewline
\center{$\vec u_{21}  $} 		&\center{$\vec u_{12} $} \tabularnewline
\multicolumn{2}{ c }{Inputs}\\
\center{$\vec u_{12}$}	& \center{$\vec u_{21}$}\tabularnewline
\multicolumn{2}{ c }{Equations}\tabularnewline
\center{$\dot {\vec x_1} = \vec f_1(\vec x_1, \operatorname{Ext}(\vec u_{12}))	$} & \center{$\dot{ \vec x_2} = \vec f_2(\vec x_2, \operatorname{Ext}(\vec u_{21}))$}	\\
\end{tabular}\\
\caption{Cosimulation scheme}
\end{center}
\end{table}
as given in table \ref{tab:generalScheme} and before in \cite[Section 3.1]{Moshagen2017} were derived. Shifting the solving of  the equations \eqref{coupling1} ff., $0=\vec h_{ij}(\vek x_j, \vek z_j, \vek u_{ij}) $ to the receiving system establishes that the  exchanged quantities $u_{ij}$ are given by some of the subsystems states, see \cite[Section 3.1]{Moshagen2017}. This restricts generality slightly, as one has to assume $\vek h_{ij}=\vek h_{ij}(\vek x_j, \vek u_{ij} )$ independent of any algebraic variable for that shifting of solving, but it simplifies the analysis.  \\
Furthermore, we have to assume that the system, if DAE, is of order 1, such that the algebraic part can be solved and $\vek u_{ij}$ can be used when evaluating $\vec f_i$, equivalently, that the system can be solved by the state space method \cite[Section 3.1]{Moshagen2017}. Such, arguments from the theory of ODE can be applied. 
This given, for  scheme \ref{tab:generalScheme} the result  \cite[theorem 3.6]{Moshagen2017} was proven: 
\begin{theorem}\label{th:convergence}
Let S be a set of ODE which is split into disjoint subsystems $S_k$ of the shape 
\begin{equation}
\dot{\vec x}_{\mathcal D_k} = f_{\mathcal D_k}(t,\left[{\vec x}_{\mathcal D_k},\vec x_{\mathcal I_k}\right] ),
\end{equation}
the $\mathcal D_k$ and $\mathcal I_k$ denoting index sets. Disjoint subsystems means that the  $\mathcal D_k$ are mutually disjoint and $\vec x_{\cup\mathcal D_k}=\vec x$, and   $\mathcal D_k$ and $\mathcal I_k$ are disjoint.
Let ${T_i}$ be a time grid with width $H$, and
let the inputs  $\vec x_{\mathcal I_k}$ of all subsystems be extrapolated at $T_i$ with polynomial order $P$ and then be solved with an one-step method of order  $p$ and maximal stepwidth $h\le H$. Then for the error of the numerical solution $\vec\epsilon_{\Delta,S}(T_{j+1})$ the estimate 
\begin{equation}\label{convErrBoundingFunction}
\norm{\vec\epsilon_{\Delta,S}(T_{j+1})} 
\le \frac{1}{L}\left(\tilde{\tilde C}h_\Delta^{p}
+CH^{P+1}\right)\left(  e^{L(T_{j+1}-T_0)}-1\right)
\end{equation}
holds.
\end{theorem}

The cosimulation scheme with balance correction reads as in Table \ref{tab:cosimSchemeBC},
\begin{table}[h!tb]\label{tab:cosimSchemeBC}
\begin{center}
\begin{tabular}{p{0.37\textwidth} | p{0.37\textwidth}} 
\multicolumn{2}{ c }{Cosimulation  scheme}\tabularnewline
\center{$S_1$} & \center{$S_2$}\tabularnewline
\multicolumn{2}{ c }{System States}\tabularnewline
\center{$\vec x_1  $} 		&\center{$\vec x_2 $} \tabularnewline
\multicolumn{2}{ c }{Outputs}\tabularnewline
\center{$\vec u_{21}  $} 		&\center{$\vec u_{12} $} \tabularnewline
\multicolumn{2}{ c }{Inputs}\\
\center{$\vec u_{12}$}	& \center{$\vec u_{21}$}\tabularnewline
\multicolumn{2}{ c }{Equations}\tabularnewline
\center{$\dot {\vec x_1} = \vec f_1(\vec x_1, \operatorname{Ext}(\vec u_{12})) +\Delta E_1^{j-1}\phi_j (t)	$} & \center{$\dot{ \vec x_2} = \vec f_2(\vec x_2, \operatorname{Ext}(\vec u_{21})) +\Delta E_2^{j-1}\phi_j (t) $}	\\
\end{tabular}\\
\caption{Cosimulation scheme with balance correction}
\end{center}
\end{table}
using upper indices for the time interval and 
\begin{equation}\label{eq:BC}
\Delta E_i^{j} := \int_{T_{j-1}}^{T_j}\vek u_i dt  - \int_{T_{j-1}}^{T_j} \overline{ \vek u}_i dt
\end{equation} 
are
added in the time step $j+1$.  
The correction that is applied in  the $j$-th interval is then
$\Delta E_i^{j-1}\phi_j (t)$.\\
For this method,
\cite[Theorem 4.1]{Moshagen2017} tells that the convergence result theorem \ref{th:convergence} still holds:
\begin{theorem}\label{th:convergenceBC}
Let $S$, $S_k$, ${\vec x}_{\mathcal D_k}$, $\vec x_{\mathcal I_k}$ and ${T_i}$ be as in theorem \ref{th:convergence}, but balance correction contributions be added to the extrapolated variables.
Then the estimate from  theorem \ref{th:convergence} still holds.
\end{theorem}
Numerical tests and considerations indicate that balance correction methods improve the convergence order of explicit cosimulation schemes by one, but no proof has been written down for this yet. 

\subsection{Linear Problem for convergence examination}

These results were confirmed numerically using the twodimensional linear problem
\begin{equation}\label{eq:LinProb}
\dot{\vec x} = \ten A \vec x \quad 
\end{equation}
which written as a cosimulation problem is
\begin{equation}\label{eq:LinProbCo}
\begin{cases}
\dot x_1 = a_{1,1}x_1 + a_{1,2}\operatorname{Ext}(x_2) 
\\ \dot x_2 = a_{2,2}x_2 + a_{2,1}\operatorname{Ext}(x_1) 
\end{cases}.
\end{equation}
The matrix entries are chosen such that 
\begin{itemize}
\item an unidirectional dependency on input is given: $a_{11},a_{21},a_{22}\neq 0$, $a_{12}=0$,
\item a mutual dependency is given: $a_{12},a_{21}\neq 0$, $a_{11}=a_{22}=0$.
\end{itemize}
Explicit and implicit extrapolation was used.
\begin{figure}
\includegraphics[ width=0.8\textwidth]{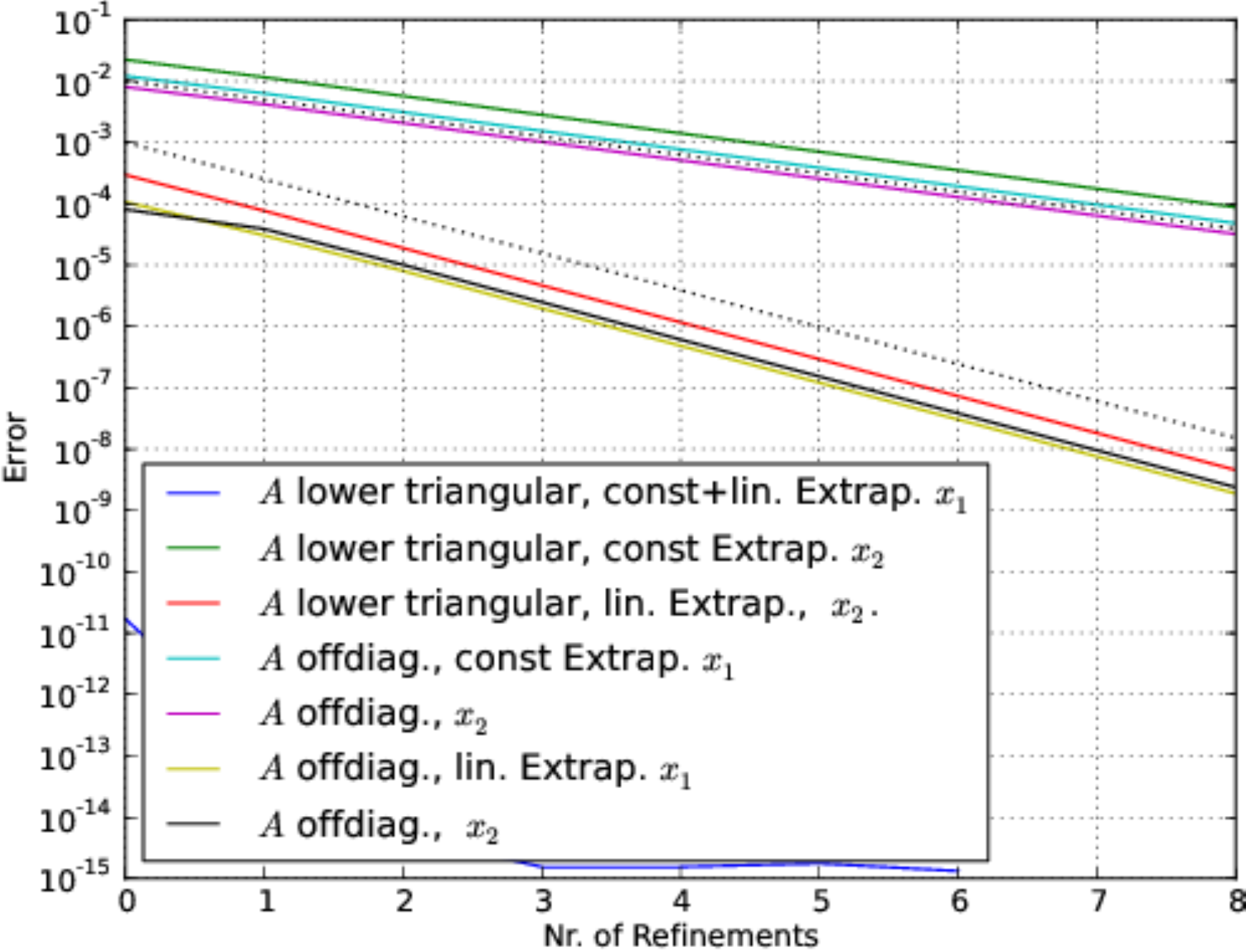}
\caption{\label{CosimConvergencePlot} Simulation of the system \eqref{eq:LinProbCo} with matrix $A\in \mathbb R^2$ realizing one-directional coupling (lower triangular shape) resp. mutual coupling (only offdiagonal entries are nozero) in the cosimulation scheme with constant and linear extrapolation, varying the exchange step size $H$.  }
\end{figure}
As ODE solver on subsystems, any solver that does not dominate the convergence and stability behavior of the cosimulation scheme could be used. Choice was \emph{vode} and \emph{zvode} from the \emph{numpy} Python numerics library, which both implement implicit Adams method if problem is nonstiff and BDF if it is.  \\

Figure \ref{CosimConvergencePlot} shows the convergence result for the four situations.   To show the predictions made by \eqref{convErrBoundingFunction} in terms of $H$, it is necessary that the subsystems methods contribution $\tilde{\tilde C} h_\Delta^{p}$ is of higher order than the extrapolation and that the method used is a one-step method, as  many multistep methods use one-step methods of unknown order at start, and that would become visible here as at each exchange step requires one restart \cite{Moshagen2017}. 
%
Thus \emph{dopri5}, an explicite Runge-Kutta method, was chosen, using the built-in stepsize control with default absolute tolerance $10^{-12}$.\\
The figure shows that convergence is of order 1 for constant extrapolation and of order 2 for linear extrapolation, as predicted by \eqref{convErrBoundingFunction}. As discussed, there is no order loss for linear extrapolation and circular dependency of inputs, but not even an higher error, in spite of the negative effects that should occur.\\ 
The error of the first component of the lower triangular, thus unidirectionally coupled system is very low as it has no extrapolation contribution, thus indicating that the error made by $dopri$ is low enough to allow for judgement of the effect of extrapolation error.

\subsection{Stability}\label{stability}

But concerning stability, in   \cite{ScharffMoshagen2017} and \cite{Moshagen2017} it is shown that cosimulation schemes are not stable for linear problems. The stability for linear problems replaces the notion of A-stability, as the one-component equation used there cannot be split. Linear stability  is shown numerically as well as by arguing that the ODE \eqref{eq:LinProbCo} which is  induced by extrapolating inputs is not stable   \cite[Section 3.5]{Moshagen2017}.

Balance correction cannot settle this issue: the linear spring-mass oscillator  
\begin{equation} \label{springMass}
\dot{ \vec{ x}}
= \ten A\vec{x} 
= \begin{pmatrix}
0 & 1\\
-\frac{c}{m} & \left(-\frac{d}{m}\right)
 \end{pmatrix} \vec x, 
 \qquad \vec{x} = 
 \begin{pmatrix}
x\\
\dot x
 \end{pmatrix}
\end{equation}
with mass $m$, spring constant $c$, and in which the damping constant $d$ shall vanish, was solved to numerically examine the stability of the method. This problem is an implementation of the linear problem \eqref{eq:LinProb} and the most simple problem possible that is linear and can be splitted, as it has the least number of components required for coupling, and coupling is its only contribution to derivatives, and this moreover linear. The eigenvalues of the uncoupled problem are purely imaginary, so the problem is stable. \\
%
%
%
\begin{table}[h!tb]
\begin{tabular}{c|c} 
Spring  & Mass\tabularnewline
\multicolumn{2}{ c }{System States}\tabularnewline
 $ x_1 := s =  x $ 		&  $ x_2 := v = \dot x $ \tabularnewline
\multicolumn{2}{ c }{Outputs}\tabularnewline
$ u_{21}:=F =-cx  $		&$ u_{12}:= v = \dot x $ \tabularnewline
    &    \tabularnewline
\multicolumn{2}{ c }{Inputs}\\
$  u_{12}$	& $ u_{21}$\tabularnewline
\multicolumn{2}{ c }{Equations}\tabularnewline
$\dot { x_1} = \operatorname{Ext}( u_{12}) = v$ & $\dot{  x_2} = -\frac{1}{m} \operatorname{Ext}(u_{21}) $\tabularnewline
                & $= -\frac{F}{m}$
\end{tabular}
\begin{tabular}{c|c}
Spring & Mass\tabularnewline
\multicolumn{2}{ c }{System States}\tabularnewline
 	$\hdots$	&  $\hdots$\tabularnewline
\multicolumn{2}{ c }{Outputs}\tabularnewline
\parbox[t]{0.25\textwidth}{$ \vec  u_{12}:=(f,\dot f)$\\ $=(-cx, -cv) $ } 		&\parbox[t]{0.25\textwidth}{ $\vec  u_{21}:=( v,a)$ \\$= (\dot x, \dot f/m )$} \tabularnewline
\multicolumn{2}{ c }{Inputs}\\
$ \vek u_{12}$	& $ \vek u_{21}$\tabularnewline
\multicolumn{2}{ c }{Equations}\tabularnewline
 $\vdots$&$\vdots$	\tabularnewline
 &\tabularnewline
\end{tabular}
\caption{\label{CosimSchemesSpringMass}Cosimulation Schemes for the spring-mass system, left constant, right linear extrapolation}
\end{table}
It is solved with the cosimulation scheme  given as above in this section, which yields  
 \ref{CosimSchemesSpringMass}.
Output of the spring is the force $F = -cx$, 
that of the mass is the velocity $v = \dot x$. 
\begin{figure}
\includegraphics[ width=0.5\textwidth]{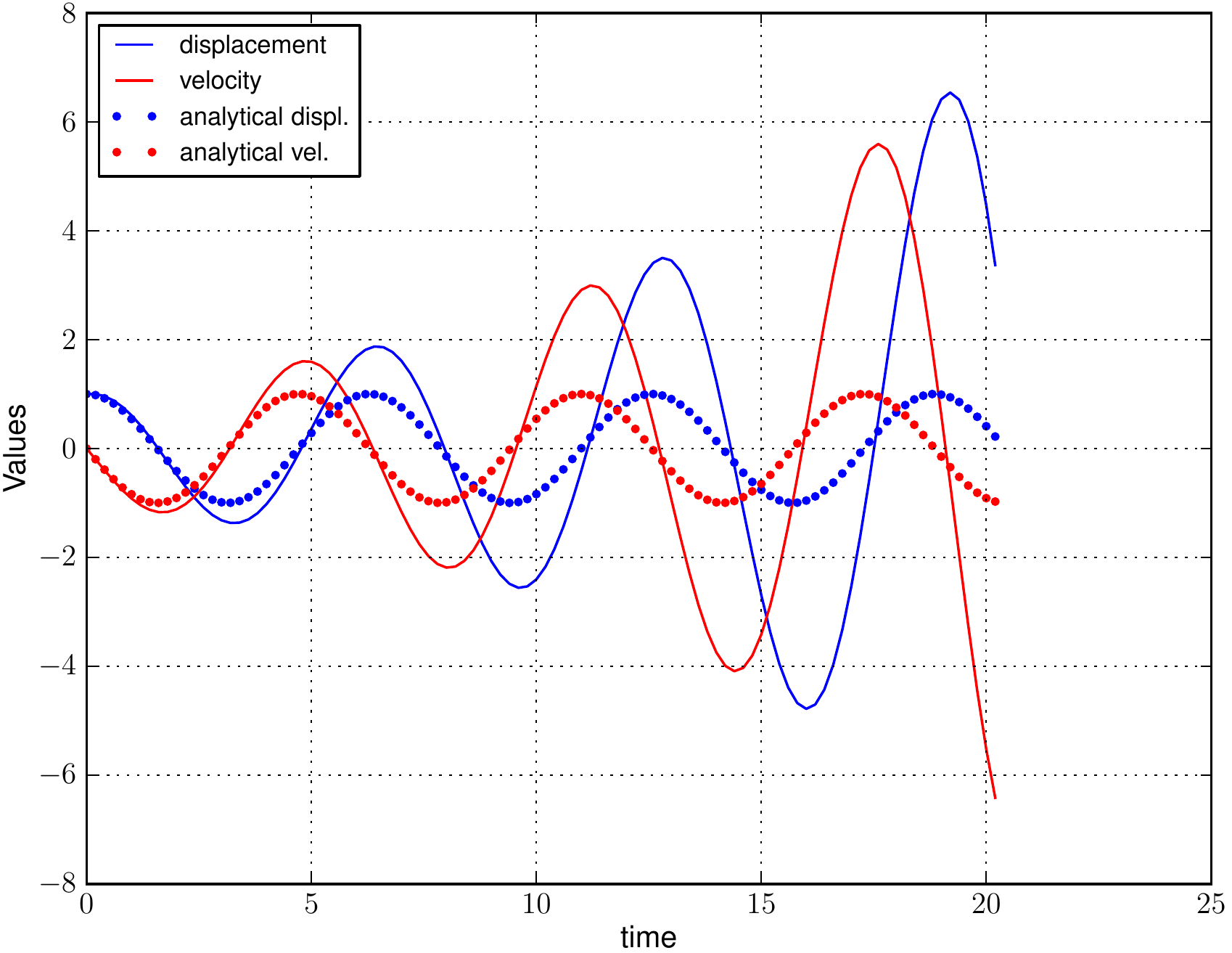}
\includegraphics[ width=0.5\textwidth]{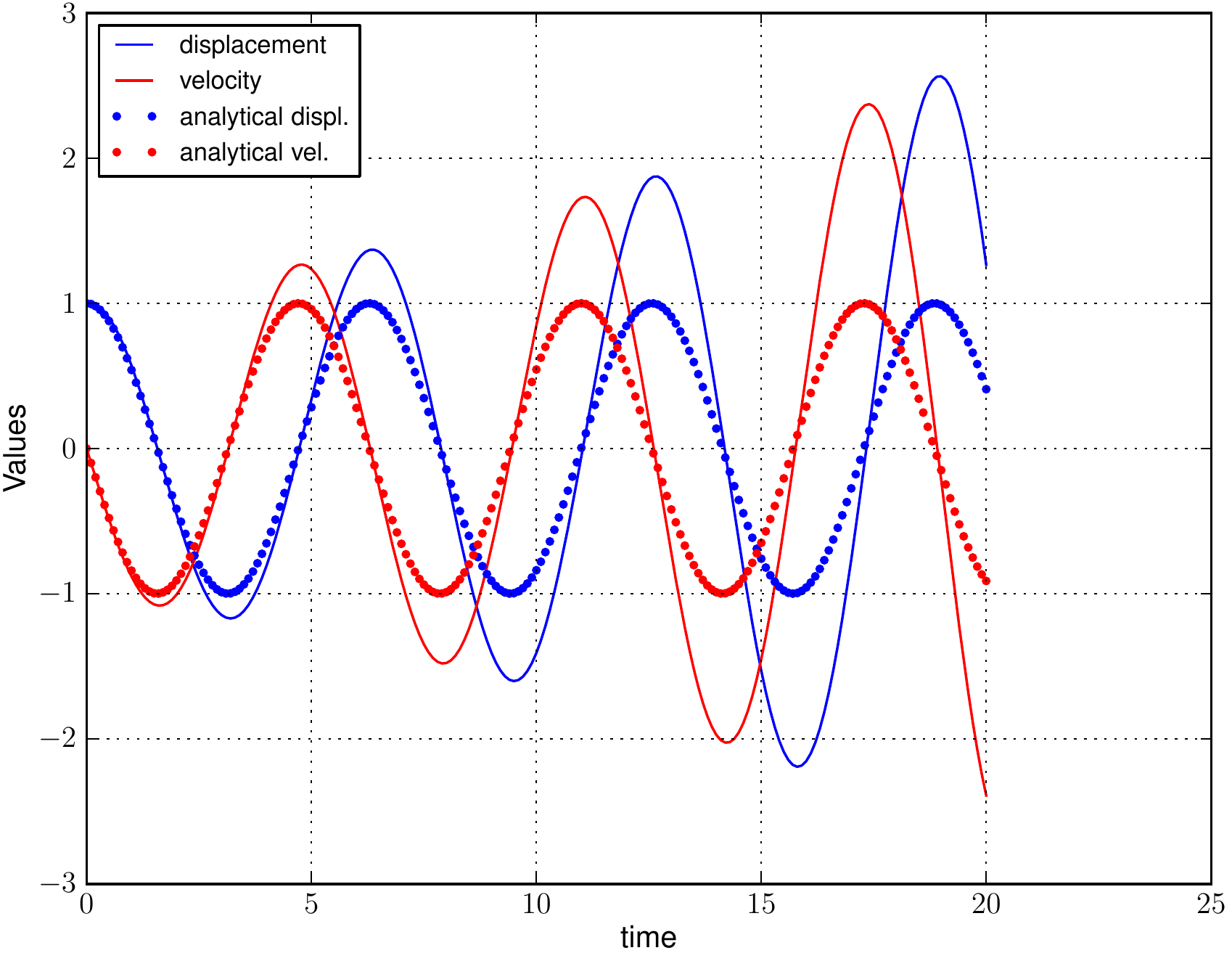}\\
\includegraphics[ width=0.5\textwidth]{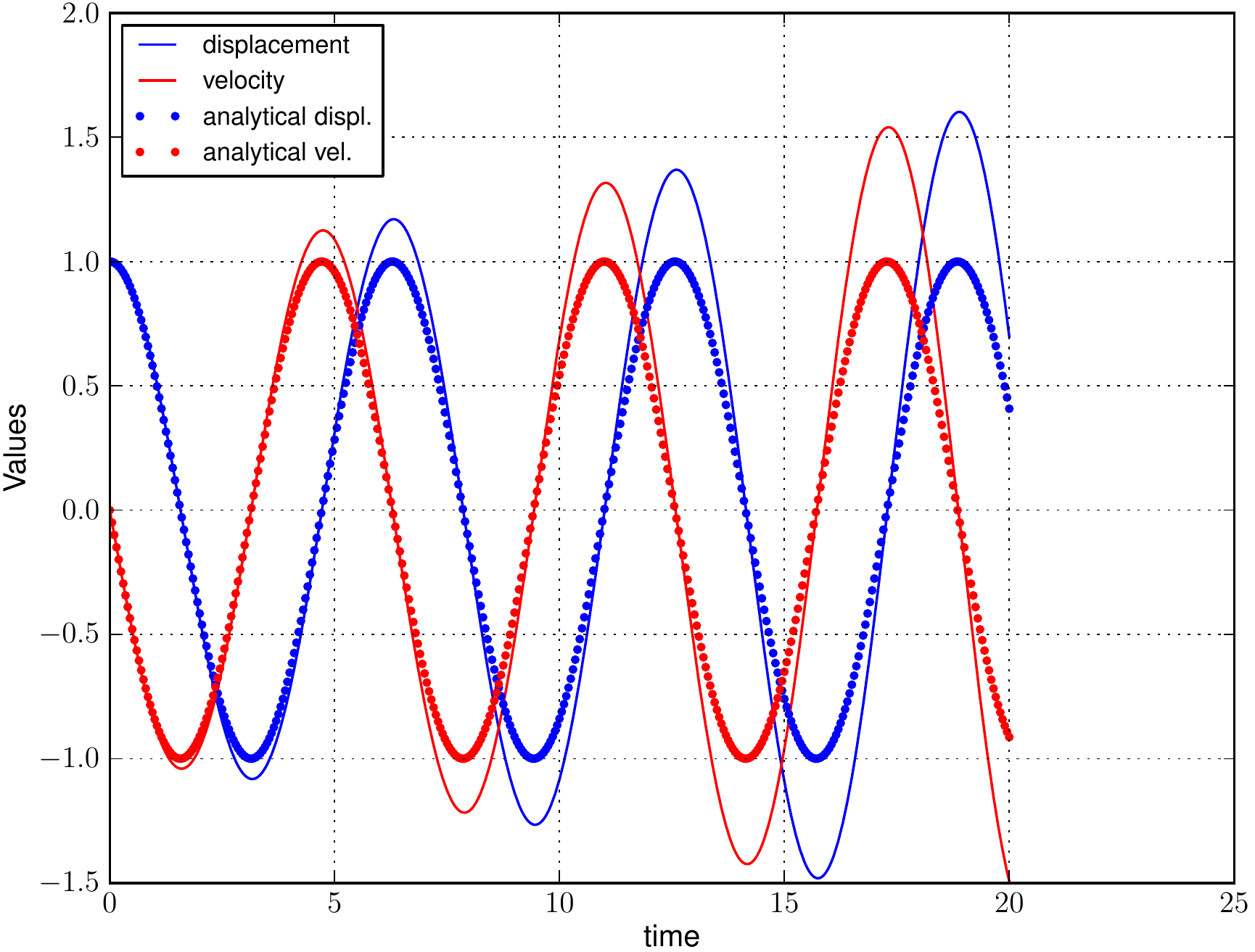}
\includegraphics[ width=0.5\textwidth]{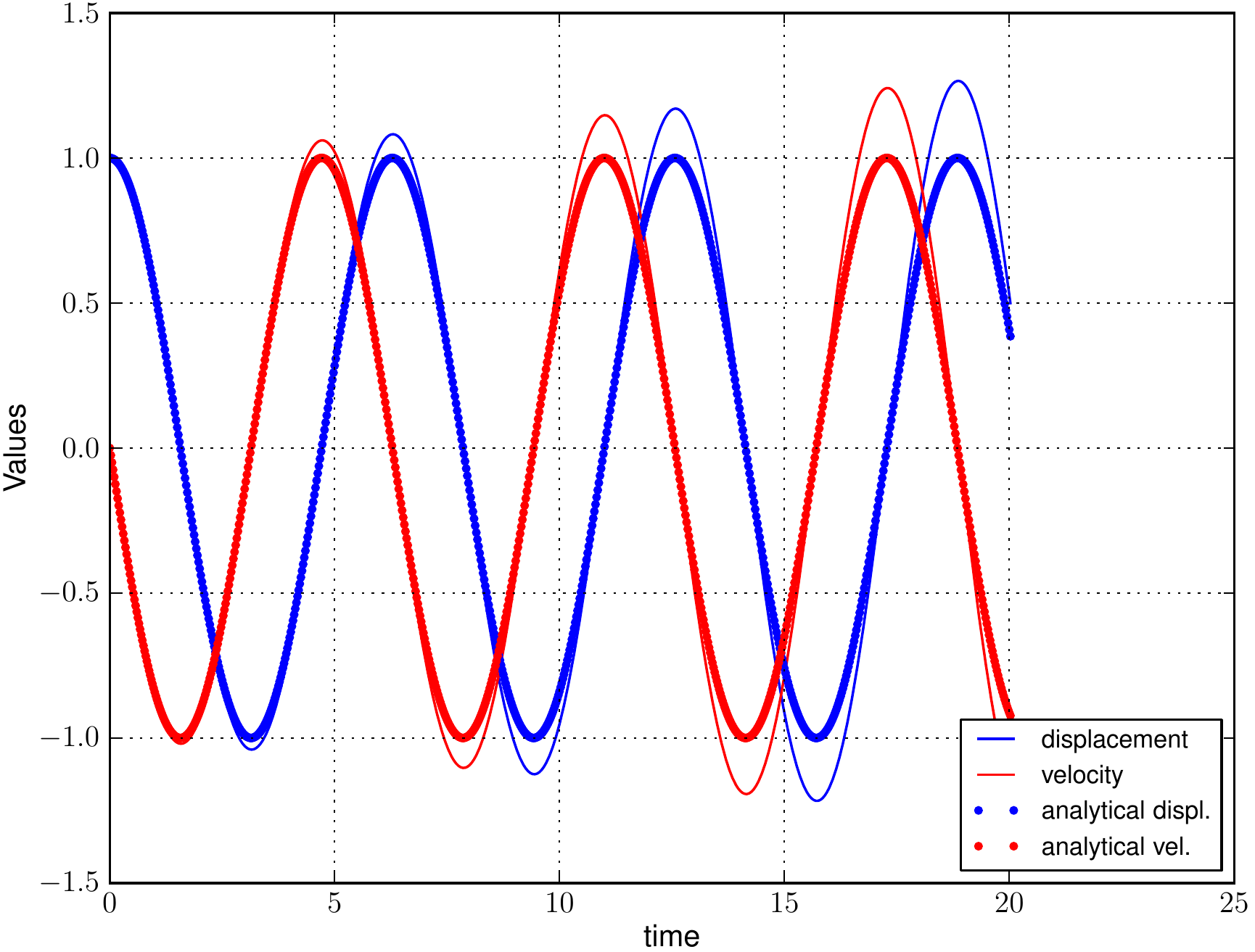}\\
\caption{\label{fig:constCosimConvergence} Simulation of the system \eqref{springMass} in the cosimulation scheme with constant extrapolation, varying the exchange step size $H$.  Upper row, left: $H = 0.2$, right:  $H = 0.1$, lower row: left: $H = 0.05$, right:  $H = 0.025$. Convergence of order $H$ is given, but there is no stability for any step size in sight. Previously published in \cite{ScharffMoshagen2017}.}
\end{figure}
\begin{figure}
\includegraphics[ width=0.5\textwidth]{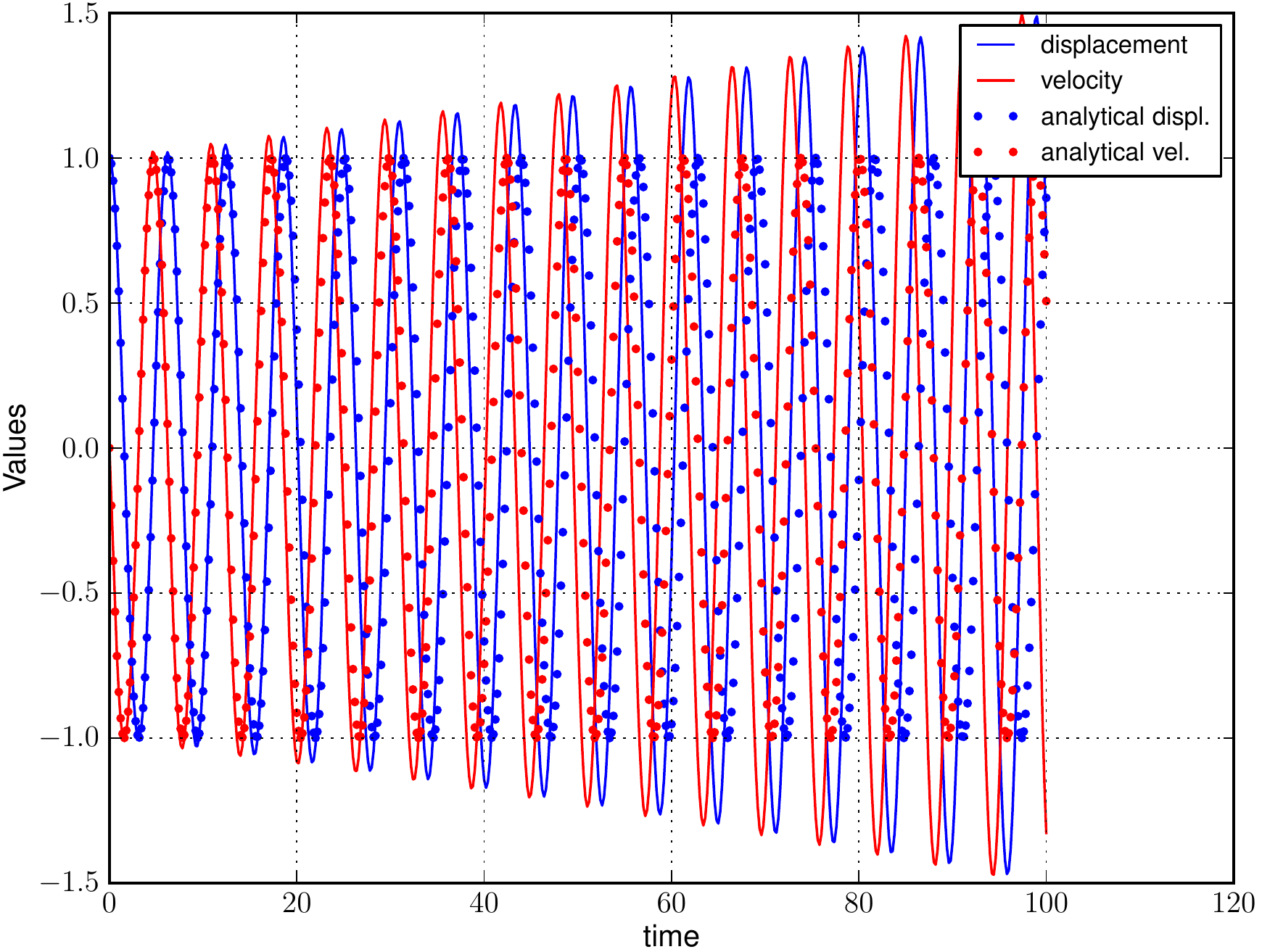}
\includegraphics[ width=0.5\textwidth]{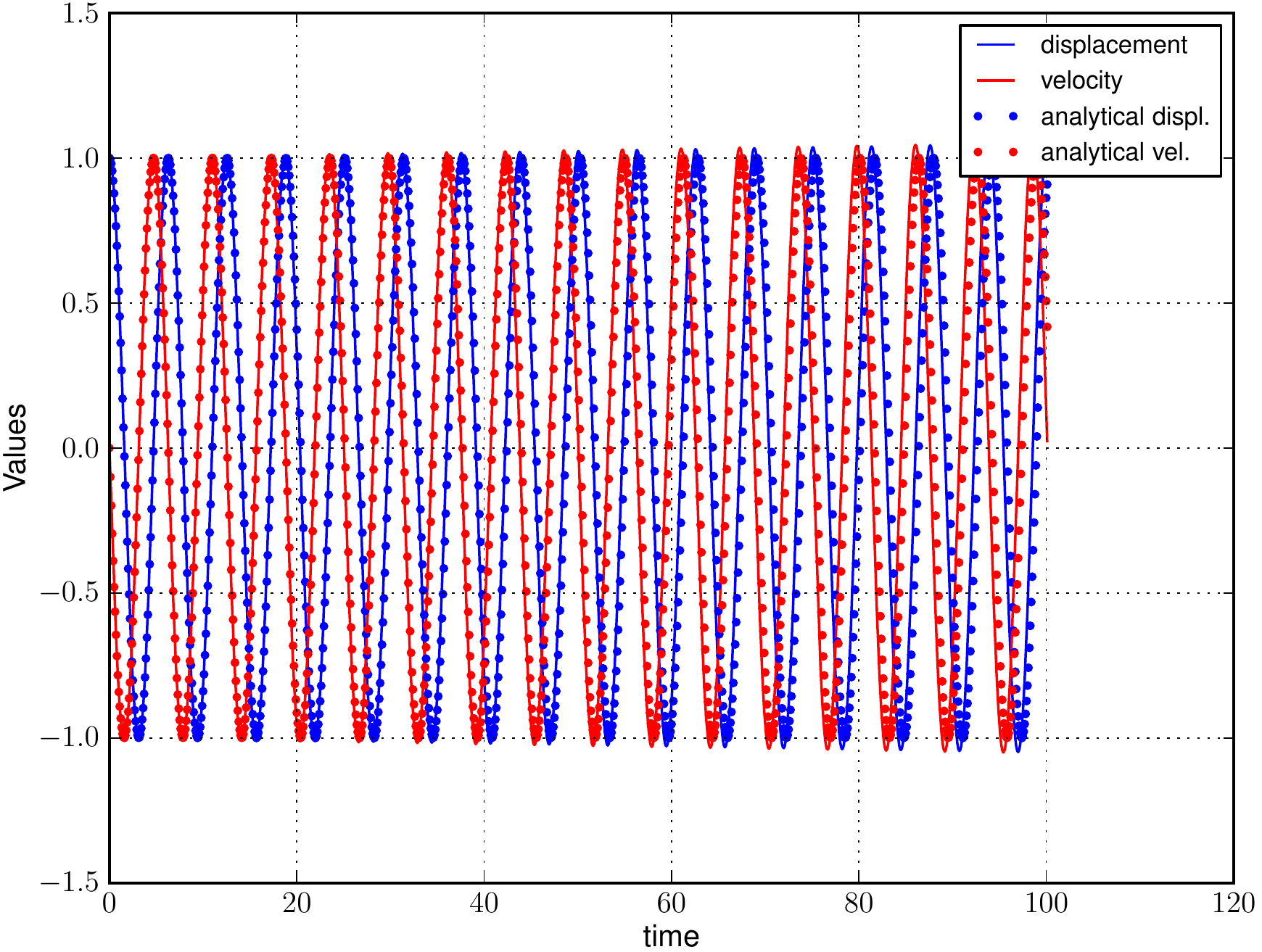}\\
\caption{\label{fig:linCosimConvergence} Simulation of the system \eqref{springMass} in the cosimulation scheme with linear extrapolation, varying the exchange step size $H$.   Left: $H = 0.2$, right:  $H = 0.1$. $H = 0.05$   $H = 0.025$ are not shown because error is not visible in plot. Convergence of order $H^2$ is given, but as in the constant extrapolation case there is no stability. Previously published in \cite{ScharffMoshagen2017}.}
\end{figure}
The numerical convergence examination backs up the results Theorem \ref{th:convergence} and \ref{th:convergenceBC}. 
But  the method is unstable for its explicit contributions, as proven in  \cite[Section 3.5]{Moshagen2017}. This means $\norm{\vec x}\longrightarrow \infty$ for $t\longrightarrow \infty$. The energy of our system is $E=\frac{1}{2}m v^2+ \frac{1}{2}cs^2=\scalar{\vec x}{\vec x}_{\frac{1}{2}\operatorname{diag}(m,c)}$, which is an equivalent norm, so lack of stability is equivalent to energy augmentation. \\  
So here, this lack of stability can be interpreted in physics 
as a consequence of extrapolation errors in factors of power acting on subsystems boundaries:
 In \cite{ScharffMoshagen2017} it was pointed out  
that errors in the force $y_1$ made during data exchange lead to errors in the power that acts on the mass. The system picks up energy and behaves unstable (See figure \ref{fig:constCosimConvergence}). 

In the following sections, it will be shown that balance correction techniques applied to the impulse as the integral of the force do not prevent this effect, as the a posteriori refeed of force then acts at another system state than it should, as the states have changed meanwhile - here the mass has changed its velocity.\\

\begin{figure}
\includegraphics[ width=0.5\textwidth]{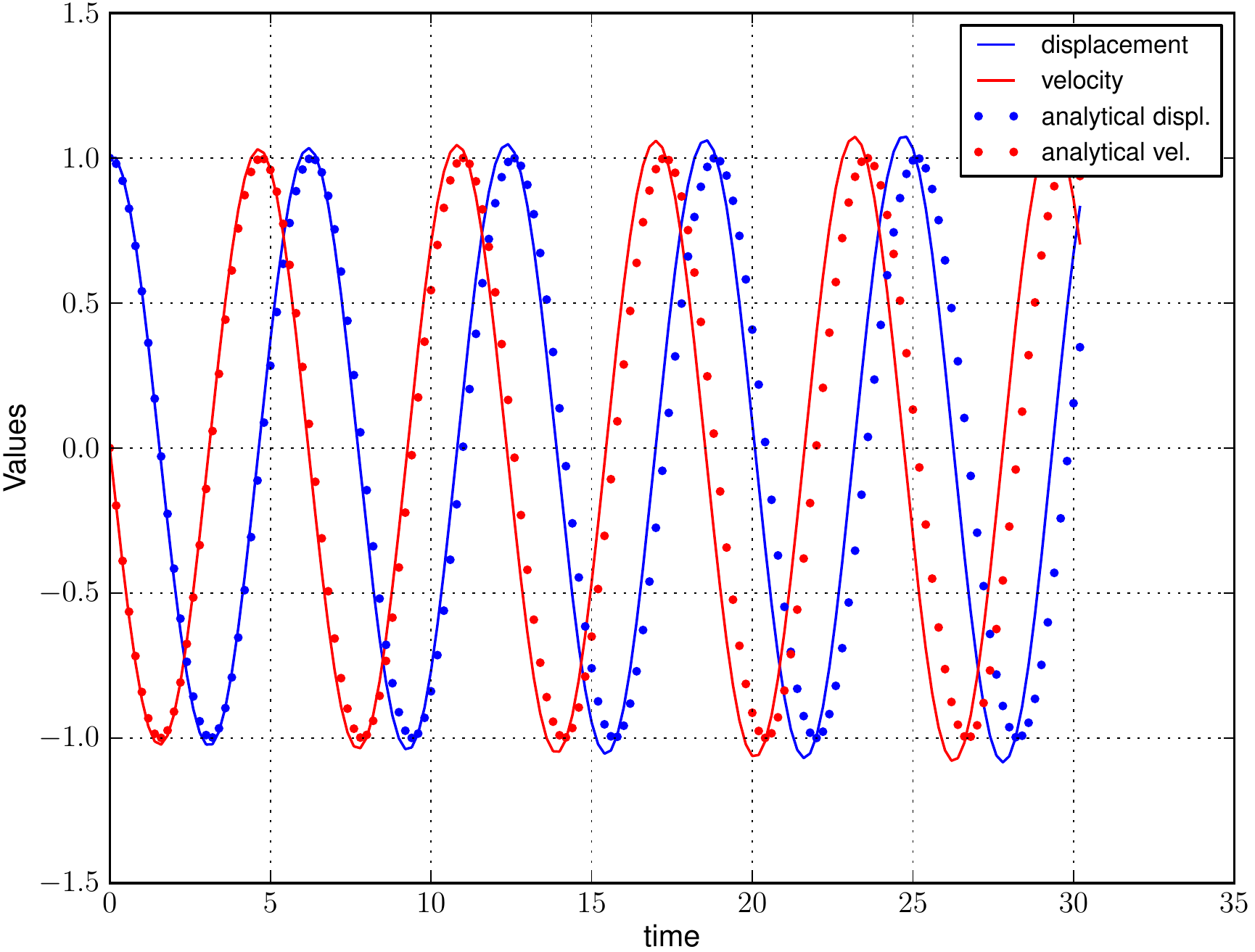}
\includegraphics[ width=0.5\textwidth]{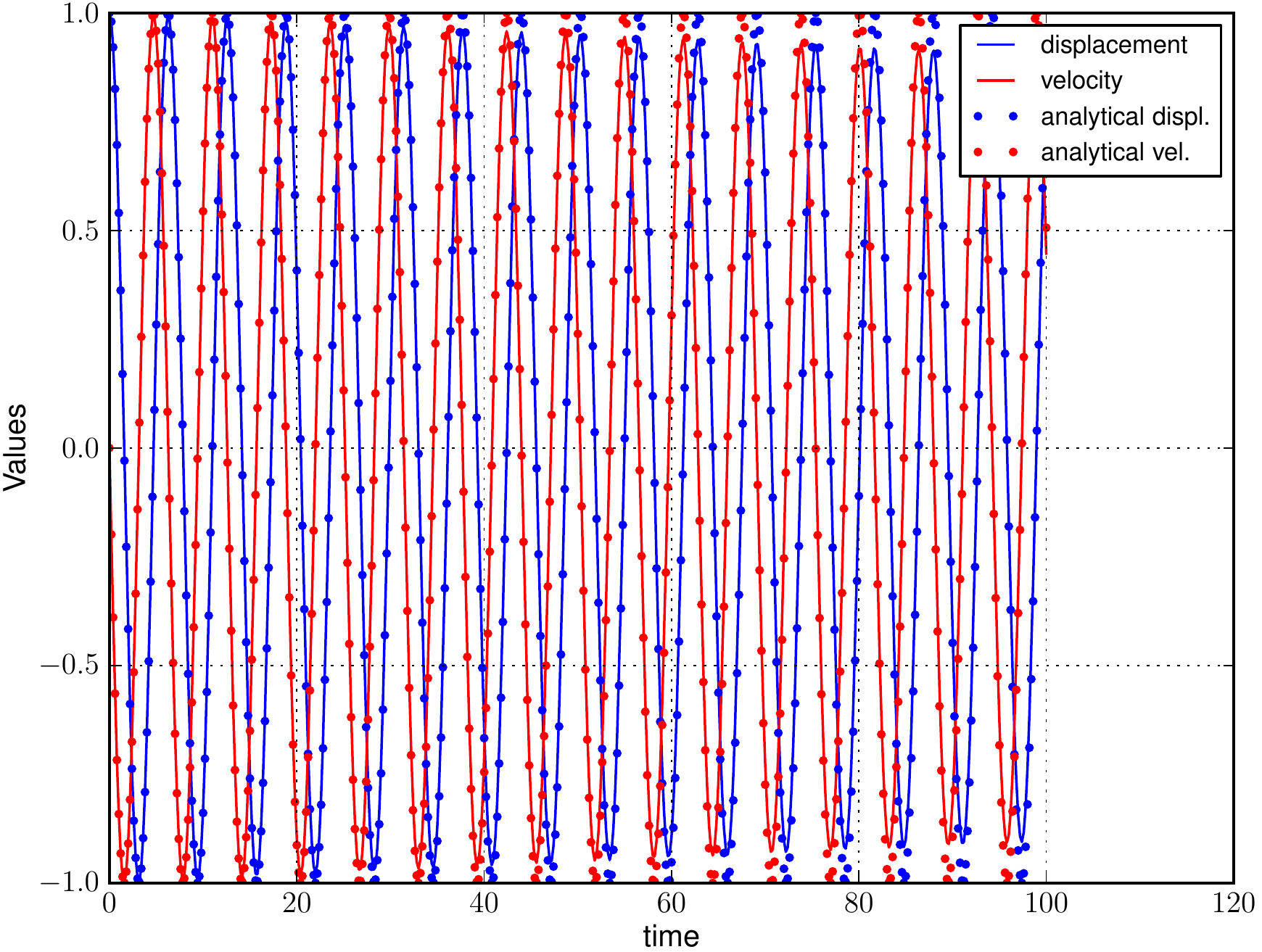}
\caption{ \label{fig:springMassBCCosim} Simulation of the system \eqref{springMass} in the cosimulation scheme with balance correction, $H=0.2s$. Left: constant extrapolation, right: linear extrapolation of received data.  The rise in energy is decelerated for the constant extrapolation (compare to Figure \ref{fig:constCosimConvergence}), but stability is not achieved. For linear extrapolation, the simulation acts as if damped (compare  to Figure \ref{fig:linCosimConvergence}) and thus is strictly speaking stable,  but energy is not conserved as it should.}
\end{figure}

\input{classificationOfBilanceErrors.tex}

\section{Enforcing  Balance by sharing the view on Potential flow}

It now is clear that balance correction methods can hardly stabilize systems that suffer from the effects described in \ref{classification} and \ref{factors} as it considers only an error made in the amount of a quantity, but the correction of amount is done at another, so wrong, time.

\subsection{The method}\label{enforcingEnergyBalance}
The key feature to establish energy balance is exchanging the value of power  and calculating the variable of interest from that power.   
Consider a cosimulation problem with subsystems $S_1$ and $S_2$ as given by equations \eqref{split1} - \eqref{coupling2} with states $\vec x_1$ and $\vec x_2$ respectively and inputs $\vec u_{21}$ and $\vec u_{12}$.  We suggest the following procedure to enforce energy balance between subsystems $S_1$ and $S_2$. 
\begin{enumerate}
\item At data exchange timepoint $T_n$ the powers $P_{ij}$ as the flux of energy are calculated in both subsystems, using up-to-date input $\vek u^n_{ij}$. In general $P_{ij}\neq P_{ji}$, although in some situations equality holds. In the input vectors $\vec u_{ij}$  one component is replaced by $P_{ij}$ and that new vector is exchanged between the subsystems. Applied to  $S_1-S2$ setting,  the values $P_{21}$ and $P_{12}$ are exchanged, this means $S_1$'s point of view about the power is passed on to $S_2$ and vice versa. The $\vec u_{12}$ and $\vec u_{21}$ are also exchanged, but one component $(u_{ij})_k$ of each $\vec u_{ij}$ is omitted as it is now calculated from $P_{ij}$.
\item Now both subsystems have the same information and thus the opportunity to draw the same conclusion on what energy exchange should be assumed. We denote this assumed energy exchange as
\begin{equation}
\hat{P}_{12}(P_{21},P_{12}) = -\hat{P}_{21},
\end{equation}
a straightforward choice is $\hat P_{21} = (P_{12}-P_{21})/2= -\hat P_{12}$, where now it is necessary to  define  flow directions: $P_{ij}$ shall be negative if it leaves $S_j$, so it is counted with opposite sign in $S_i$.\\
The former input $(u_{12})_k(t)$ is calculated subject to 
\begin{equation}\label{extrapolationU_i}
P_{21}(\vec x_1(t), \vec u_{12\setminus k},(u_{12})_k(t)) = \operatorname{Ext}(\hat P_{12})
 \end{equation}
 Indexing is like follows: $P_{21}$ is the power calculated in $S_1$ for passing to $S_{2}$, calculated using $\vec u_{12}$, the input into $S_1$. 
 Analogously  $(u_{21})_k (t)$ s.t. $P_{12}(\vec x_2, \vec u_{21\setminus l}, (u_{21})_l) = \operatorname{Ext}(\hat P_{21})$ is calculated. This  requires that the maps $P_{ij}(.,.,(u_{ji})_k)$ are monotone. 
The expression ${12\setminus k}$ in index is to say that the $k$-th component of the vector is left out.

Now it is established that the inputs of $S_1$ and $S_2$ are consistent in terms of energy conservation for all $t$.\\
\end{enumerate}

\paragraph{Inversion of $P$ and its Notation} The solution of finding $(\vec u_{12})_k(t)$ subject to $P_{12}(\vec x_1(t), \vec u_{12\setminus k},(u_{12})_k(t)) = \operatorname{Ext}(\hat P_{12})$ will in the following be denoted
as 
\begin{equation}
(u_{12})_k(t) =P_{21}^{-\vec u_{k}}(\operatorname{Ext}(\hat P_{12}),\vec x_1, \vec u_{12\setminus k} )
\end{equation}
in analogy to the usual denotion of inverse functions by exponent  $\cdot^{-1}$. This inversion in fact can be ill-conditioned in practice.\\

Precisely, the cosimulation scheme\\[11pt]
\begin{center}
\begin{tabular}{c|c} 
$S_1$ & $S_2$\tabularnewline
\multicolumn{2}{ c }{System States}\\ 
$\vec x_1  $ 		&$\vec x_2 $ \tabularnewline[11pt]
\multicolumn{2}{ c }{Outputs}\tabularnewline
$ (\vec u_{21}, \dot{\vec u}_{21})$ 		&$ (\vec u_{12},\dot{\vec u}_{12})$ \tabularnewline[11pt]
\multicolumn{2}{ c }{Inputs}\\
$ (\vec u_{12},\dot{\vec u}_{12})$ &
$ (\vec u_{21}, \dot{\vec u}_{21})$ \tabularnewline[11pt]
\multicolumn{2}{ c }{Equations}\tabularnewline
$\dot {\vec x_1} = \vec f_1(\vec x_1, \operatorname{Ext}(\vec u_{12}))	$ & $\dot{ \vec x_2} = \vec f_2(\vec x_2, \operatorname{Ext}(\vec u_{21}))$	\\
\end{tabular}
\end{center}
is replaced by
\begin{center}
\begin{tabular}{c | c}
$S_1$ & $S_2$\tabularnewline
\multicolumn{2}{ c }{Outputs}\tabularnewline
$ \vec u_{21, \text{Std}}:= $               & $\vec u_{12,\text{Std}}:= $\tabularnewline
$((\vec u_{21})_1, ...(\vec u_{21})_n)$     & $((\vec u_{12})_1, ...(\vec u_{1 2})_m)$\tabularnewline
$P_{21}= P_{21}(\vec x_1, \vec u_{12}) ) $ & $P_{12}= P_{12}(\vec x_2, \vec u_{21})) $ \tabularnewline[11pt]
\multicolumn{2}{ c }{Inputs (without loss of generality)}\\
$\vec u_{12}:=$	&$\vec u_{21} :=  $\tabularnewline
$( (\vec u_{12})_1, ...(\vec u_{1 n})_{ m-1},  P_{12}(\vec x_2, \vec u_{21}))$	&$ ((\vec u_{21})_1, ...(\vec u_{21})_{ n-1}, P_{21}(\vec x_1, \vec u_{12}) )$\tabularnewline[11pt]
\multicolumn{2}{ c }{Variables depending on Inputs}\\
$\hat P_{12}(P_{21},P_{12})$ & $-\hat P_{12}(P_{21},P_{12})$\tabularnewline
$\dot{\hat P}_{12}(\dot P_{21},\dot P_{12})$ & $-\dot{\hat P}_{12}(\dot P_{21},\dot P_{12})$\tabularnewline
$(\vec u_{12})_m(t)$  
s.t. 

	&  $(\vec u_{21})_n(t)$ s.t. \tabularnewline
 $ P_{21}(\vec x_1, \vec u_{12\setminus m},(\vec u_{12})_m)(t) = \operatorname{Ext}(\hat P_{12})$
	&   $ P_{12}(\vec x_2, \vec u_{21\setminus n},(\vec u_{12})_n)(t) = \operatorname{Ext}(\hat P_{12})$
    \tabularnewline[11pt]
\multicolumn{2}{ c }{Equations}\tabularnewline
$\dot {\vec x_1} = \vec f_1(\vec x_1, \vec u_{12, \text{Std}})	$ & $\dot{ \vec x_2} = \vec f_2(\vec x_2, \vec u_{21, \text{Std}})$\\
\end{tabular}
\end{center}
Energy balance still holds when  the method is extended to more than two subsystems as balance holds at each inter-subsystem boundary.
Again, the expression $\vec u_{ij\setminus m}$ in index is to say that the $m$-th component of the vector $\vec u_{ij}$ is left out.

\subsubsection{Idle negotiation in two subsystems situation} Given two subsystems with one interface between them, from physics of course $P_{21}(\vec x_1, \vec u_{12}(x_2))= - P_{12}(\vec x_2, \vec u_{21}(x_1))$ holds. In fact,  $P_{12}$ and $P_{21}$ depend on the same variables $x_1$ and $x_2$ during cosimulation, and although the two powers may use different formula, they give equal result. In this case, the negotiating step begins with equal views and so is idle.

\subsection{Example}\label{exampleNegPower}
To apply the scheme given in  \ref{enforcingEnergyBalance} above to the model of a spring-mass system as \eqref{springMass},  replacing the standard cosimulation scheme from table \ref{CosimSchemesSpringMass}, one first calculates the energies of the systems parts, powers acting on subsystems boundaries
, and their derivatives. As $P_i=\dot W_i$, $P_i<0$ indicates that energy leaves $S_i$.\\
  
\begin{tabular}{p{0.45\textwidth} | p{0.45\textwidth}}
\center{Spring} &\center{ Mass}\tabularnewline[11pt]
\multicolumn{2}{ c }{Energy}\tabularnewline
\center{$W=\int -f \,ds=\int -f v\,dt$ }		
&\center{$W=\frac{1}{2}mv^2=\int fds =\int ma\,ds = \int mav\,dt$ }\tabularnewline[11pt]
\multicolumn{2}{ c }{Power}\\
 \center{$P=\dot W = -f v = cxv$}	& \center{$P = \dot W = mav=fv$}\tabularnewline[11pt]
\multicolumn{2}{ c }{Derivative of Power}\tabularnewline
 \center{$\dot P=c(v^2+sa) $	} & \center{$\dot P =  m(a^2 +v\dot a)= m(a^2 +v\frac{\dot f}{m})$}
\end{tabular}\\

$\dot f$ is available as output of spring, as usually serves as derivative of input.
Now the following systems are treated:\\
\begin{tabular}{p{0.45\textwidth} | p{0.45\textwidth}} 
\center{Spring} & \center{Mass}\tabularnewline
\multicolumn{2}{ c }{System States}\tabularnewline
\center{ $ x_1 := s =  x $ } 		&\center{$ x_2 := v = \dot x $} \tabularnewline
\multicolumn{2}{ c }{Outputs}\tabularnewline
\center{$ (\vec u_{21,\text{Std}})_{1}:=f =-cx  $} 		&\center{$ (\vec u_{12,\text{Std}})_{1}:= v = \dot x $} \tabularnewline
\center{$ (\vec  u_{21, \text{Std}})_{2}:=\dot f =-cv  $} 	&\center{$ (\vec u_{12,\text{Std}})_2:=  \dot v=f/m $} \tabularnewline[11pt]
\multicolumn{2}{ c }{(intermediately exchanging $\vec u_{ij,\text{Std}}$)}\tabularnewline
\center{$(\vec u_{21})_1 =P(x_1,\vec u_{12}) = cxv = cx_1(\vec u_{12})_1 $} 		&\center{$(\vec u_{12})_1  = P(x_2,\vec u_{21}) = fv =x_2(\vec u_{21})_1$}
 \tabularnewline
\center{$(\vec u_{21})_2= \dot P(x_1,\vec u_{12}) =c(v^2+xa)=c((\vec u_{12})_1^2 + x_1(\vec u_{12})_2) $ }	&\center{$(\vec u_{12})_2 =\dot P(x_2,\vec u_{21})=m(a^2 +v\frac{\dot f}{m})=m\left(\frac{(\vec u_{21})_{1}}{m}^2 + x_2\frac{(\vec u_{21})_{2}}{m}\right) $} 
\tabularnewline[11pt]
\multicolumn{2}{ c }{Inputs}
\tabularnewline
\center{$  (\vec u_{12})_1:=  \hat P$}	& \center{$ (\vec u_{21})_1 := -\hat P$}\tabularnewline
\center{$  (\vec u_{12})_2:=  \hat{ \dot{P}}$}	& \center{$ (\vec u_{21})_2 := -\hat {\dot{P}}$}
\tabularnewline[11pt]
\multicolumn{2}{ c }{Variables (inputs of standard method $\vec u_\text{std}$ ) depending on Inputs (Power)}\\
\center{$ v = \frac{\operatorname{Ext}(\hat P)}{cs}= \frac{\operatorname{Ext}(\vec u_{12})_1}{cx_1}$} 
&\center{$f = -\frac{\operatorname{Ext}(\hat P)}{v} = \frac{\operatorname{Ext}(\vec u_{21})_1}{x_2}$ }
\tabularnewline[11pt]
\multicolumn{2}{ c }{Equations}
\tabularnewline
\center{$\dot { x}_1  = v	$} & \center{$\dot{  x}_2  = \frac{f}{m}$}	
\tabularnewline
\end{tabular}\\
Remark that here 
\begin{align}
P_{21}^{-v}    &=\frac{\operatorname{Ext}(\hat P)}{cs}
\\P_{12}^{-f}  &=\frac{\operatorname{Ext}(\hat P)}{v}
\end{align}
have to be calculated. For exact values, those are well defined and bounded as $P\longrightarrow 0$ if $s\longrightarrow 0$ and  if $v\longrightarrow 0$. On a computer, they are neither defined nor bounded.  One has to switch to d'Hopitals rule for calculation near denominators zeros.

\input{stabilityOfPowerBalancedSchemes.tex}
\input{numResultsEnergyII.tex}

\section{Discussion, Conclusion and future work}
\label{sec-5}
The suggested method establishes balance of energy possibly at the cost of other balances. If we think of a system that is highly damped, this does not make sense. But if a system is undamped, the method enables applying cosimulation methods and implements an explicit but stable method. Moreover, the method has a clear interpretation in physics and can be implemented by anyone with understanding of the systems he wants to couple. For simulations in industrial research and  development, the new method is a big step forward. \\
 A future task could be to design explicit methods whose extrapolation is such that a variable connected to stability -- as energy -- is conserved.

\section{*Acknowledgments}
The author thanks Dirk Scharff for putting up the questions that led to this work.

\bibliography{../paper/ifacconf}
\bibliographystyle{nMCM}

\end{document}

%% file: abstractEnergy.tex
\begin{abstract}                
In engineering, it is a common desire to couple existing simulation tools together into one big system by passing  information from subsystems as parameters into the subsystems under influence. As executed at fixed time points, this data exchange gives the global method a strong explicite component, 
and as flows of conserved quantities are passed  across subsystem boundaries, it is not  ensured that systemwide balances are fulfilled:   the system is not solved as one single equation system. 
These \emph{balance errors} can accumulate and make simulation results inaccurate. Use of  higher-order extrapolation in exchanged data can reduce this problem but cannot solve it.\\
The remaining balance error has been handled  in past work with balance correction methods which compensate these errors by adding corrections for the balances to the signal in next coupling time step. Further past work combined smooth extrapolation of exchanged data and balance correction.
This gives rise to the problem that establishing balance of one quantity \emph{a posteriori} due to the time delay in general cannot establish or even disturbs the balances of quantities that depend on the exchanged quantities, usually energy. In this work, a method is suggested  which allows to choose the quantity that should be balanced to be that energy, and to accurately balance it.
\end{abstract}

%% file: introductionEnergy.tex
\section{Introduction}
\label{sec-1}

Engineers are increasingly relying on numerical simulation techniques. Models and 
simulation tools for various physical problems have come into existence in the past 
decades. The desire to simulate a system that consists of well described and treated 
subsystems by using appropriate solvers for each subsystem and letting them exchange the data that forms the mutual 
influence is immanent. \\ 
The situation usually is described by two coupled differential-algebraic systems $S_1$ and $S_2$ that together form a system $S$:
\begin{align}
S_1: \quad \nonumber \\ 
\dot{\vek x}_ 1 &= \vek f_1(\vek x_1,\vek x_2,\vek z_1, \vek z_2)\\
0 &= \vek g_1(\vek x_1, \vek x_2, \vek z_1, \vek z_2) \\
S_2: \nonumber \quad\\
\dot{\vek x}_2  &= \vek f_2(\vek x_1,\vek x_2,\vek z_1, \vek z_2)\\
0 &= \vek g_2(\vek x_1, \vek x_2, \vek z_1, \vek z_2). 
\end{align}
The $(\vek x_1,\vek x_2)$ are the differential states of $S$, their splitting into $\vek{ x}_i$ determines the subsystems $S_i$ together with the choices of the $\vek{z}_i$.
In \emph{Co-Simulation} the immediate mutual influence of subsystems 
 is  replaced by exchanging data at fixed time points 
 and subsystems
are solved separately and parallely but using the received parameter:
\begin{align}
S_1: \nonumber \quad\\
\dot{\vek x}_ 1 &= \vek f_1(\vek x_1, \vek z_1, \vek u_{12})\label{split1} \\ 
0 &= \vek g_1(\vek x_1,  \vek z_1, \vek u_{12}) \label{splitAlgebraic1} \\ 
S_2:\quad \nonumber \\
\dot{\vek x}_2 &= \vek f_2(\vek x_2,\vek z_2, \vek u_{21})\\
0 &= \vek g_2(\vek x_2, \vek z_2,  \vek u_{21})  \label{splitLast}
\end{align}
where $\vek u_i$ are given by coupling conditions that have to be fulfilled at exchange times $T_k$ 
\begin{align}
\vek 0 &=\vek h_{21}(\vek x_1, \vek z_1, \vek u_{21})\label{coupling1}  \\ 
\vek 0 &=\vek h_{12}(\vek x_2, \vek z_2, \vek u_{12})  \label{coupling2}
\end{align}
and are not dependent on subsystem $i$'s states any more, so are mere parameters between exchange time steps.\\
Full row rank of $d_{\vek z_i} \vek g_i$ can be assumed, such that the differential-algebraic systems are of index 1. This description of the setting is widespread (\cite{ArnoldGuenther2001}). 
With the $\vek h_{ij}$ being solved for $\vek u_{ij}$ inside the $S_j$ (let solvability be given), for systems with more than two subsystems it is more convenient to write  output variable $\vec y_j$  
and now redefine $\vek u_{ij}$ as the input of $S_i$, consisting of some components of the outputs $\vek y_j$ 
\cite{ArnoldClaussSchierz2013}. This structure is defined as kind of a standard for connecting simulators for cosimulation  by the \emph{Functional Mockup Interface} Standard \cite{FMI}. It defines clearly what information a subsystems implementation provides. \\ 
In 
Co-Simulation the variables establishing the mutual influence of subsystems 
 are exchanged at fixed time points.
 This results in continuous variables being approximated by piecewise constant
  extrapolation, as shown in the following picture: \\
\begin{figure}[h!tb]
\begin{center}
\resizebox{0.5 \textwidth}{!}{
\includegraphics{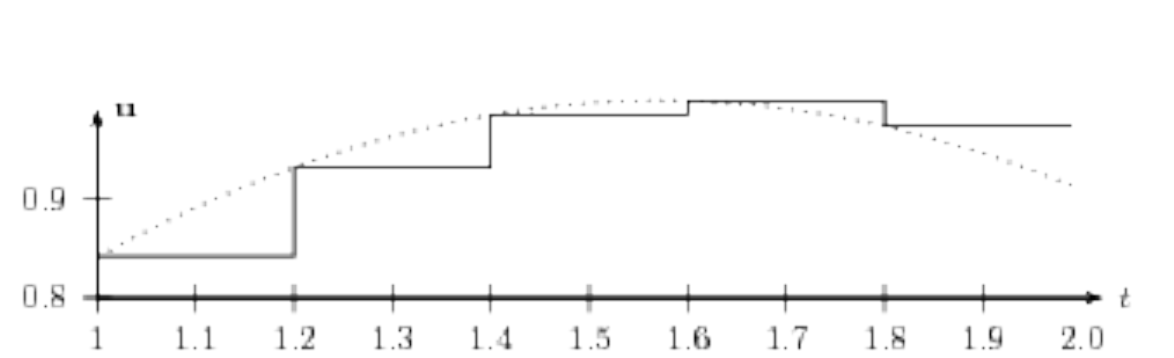}
}\end{center}
\caption{Constant extrapolation of an input signal}
\end{figure}

If one does not want to iterate on those inputs by restarting the simulations using the newly calculated inputs, one just proceeds to the next timestep. \\
This gives the calculations an explicite component, the mutual influence is now not immediate any more, inducing the typical stability problems, besides the
  approximation errors.\\
 \begin{figure}[h!tb]
\begin{center}
\resizebox{0.7\textwidth}{!}{
\includegraphics{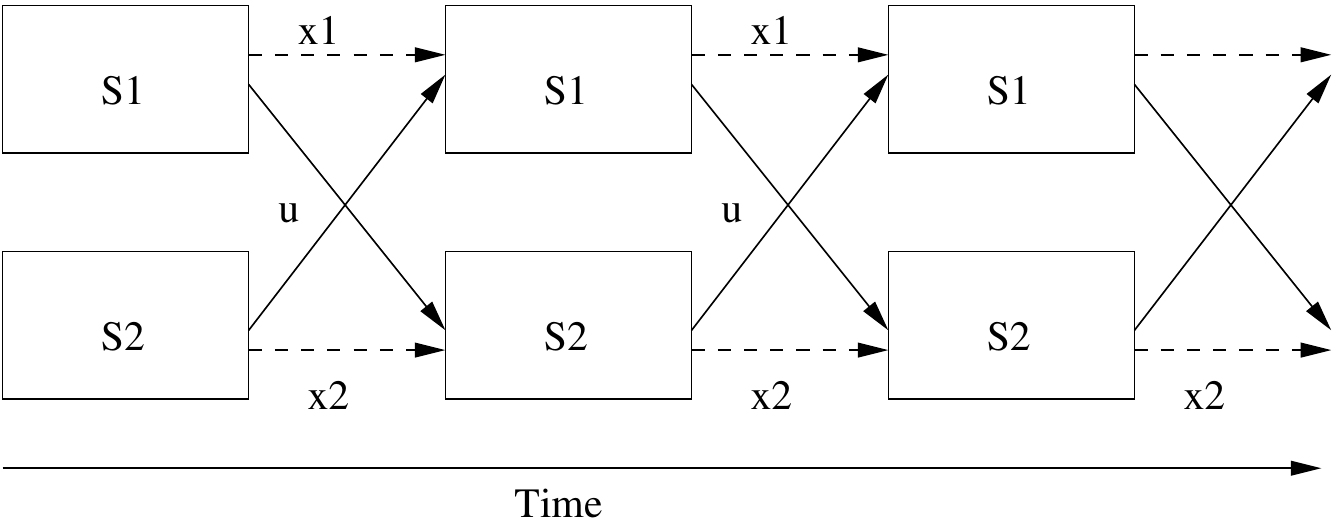}
}\end{center}
\caption{Explicit Cosimulation Scheme}
\end{figure}
But for good reasons, explicit co-simulation is a widely used method: 
It allows to put  separate submodels, for each of which a solver exists, together into one system  and simulate that system by simulating each subsystem with its   specialised solver - examination of mutual influence becomes possible without rewriting everything into one system, and   simulation speed benefits from the parallel calculation of the submodels. Usually it is highly desirable that a simulation scheme does not require repeating of exchange time intervals or iteration, as for many comercial simulation tools this would  already require  too deep intrusion into the subsystems method and too much programming in the coupling algorithm.\\

The following fields of work on explicite co-simulation can be named to be the ones of most interest:
\begin{enumerate}
\item Improvement of the approximation  of the exchanged data will most often improve simulation results \cite{Busch2012}. This  is usually done by \emph{higher-order extrapolation} of exchanged data, as shown in this plot, where  the function plotted with dots is linearly extrapolated:\\ 
\begin{figure}[h!tb]
\begin{center}
\resizebox{0.5 \textwidth}{!}{
\includegraphics{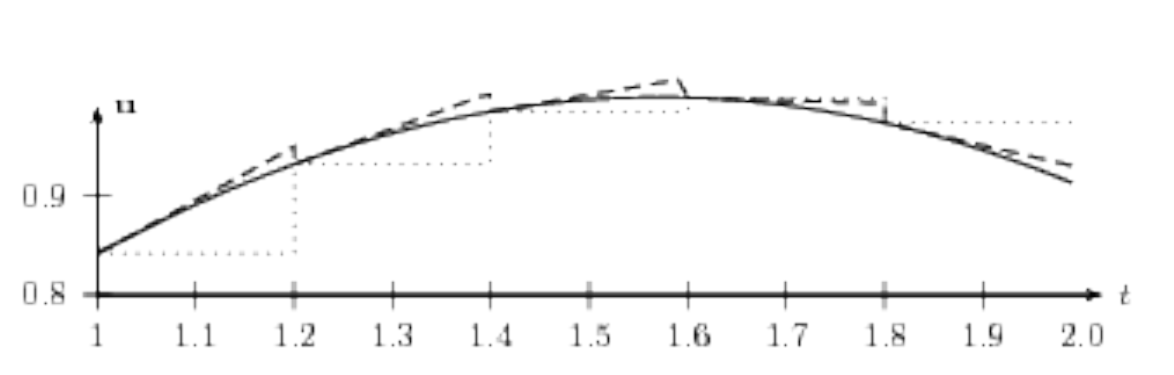}
}\end{center}
\caption{Linear extrapolation of an input signal}
\end{figure}
%

\item When the mutual influence between subsystems consists of flow of conserved quantities like mass or energy, it turns out that the improvement of the approximation of this influence by extrapolation of past data is not sufficient to establish the conservation of those quantities with the necessary accuracy.
The error that arises from the error in exchange adds up over time and
 becomes obvious (and lethal to simulation results many times). 
 In a cooling cycle example (\cite[Section 6.3]{KosselDiss}),  a gain of 1.25\% in coolant mass occurs when simulating a common situation.
\begin{figure}[h!tb]
\begin{center}
\resizebox{0.5 \textwidth}{!}{
\includegraphics{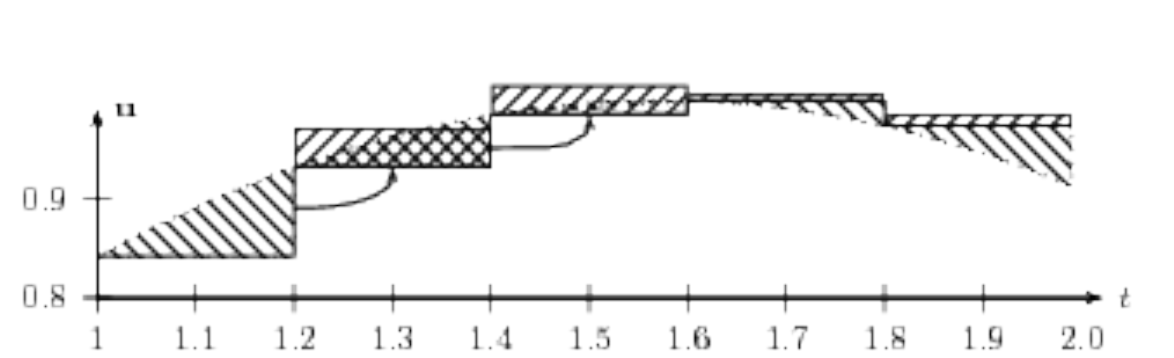}
}
\end{center}
\caption{Constant extrapolation of an input signal, balance error and its recontribution}
\end{figure}
It has been tried  to meet this challenge by   passing the amount of exchanged quantity for the past timestep along with the actual flow on to the receiving system, where then the error that has just been  commited is calculated  and added to the current flow to compensate the past error. For well damped example problems in fluid circles this method has fulfilled the expectations \cite{Scharff12}. It has been labelled \emph{balance correction}. 
\item There is good reason to prevent jumps in exchanged data by \emph{smoothing}. Higher order extrapolation polynomials cannot make extrapolated  data at the end of the exchange timestep match the newly given value.
\end{enumerate}

%% file: classificationOfBilanceErrors.tex
\section{Reclassification of Balance Errors}\label{classification}
\subsection{ Errors in balances while exchanging conserved Quantities}
These errors are defined as those  that can be calculated as   $\Delta E_{u_{ij}}^{k} = \int_{t_{k-1}}^{t_k}\vek u_{ij} dt  - \int_{t_{k-1}}^{t_k} \overline{ \vek u_{ij}}^k dt$, thus    those resulting directly from the extrapolation error in flow of the exchanged conserved quantity.
Classical balance errors of extrapolations of $k$-th order arise on intervals where $dt^{k+1} f_m$ does not change signs. The balance error is often partially compensated during passing through intervals with positive and negative sign. This does not imply it is small during all intervals of the simulation.\\
 As a typical such simulation situation think of an automotive driving cycle: Using piecewise constant extrapolation of exchanged data, there usually is a loss in conserved quantity after the phase of rising system velocity has passed - this loss remains uncompensated during the ( typically significant) phase of elevated speed, see e.g. \cite{KosselDiss}. This originally motivated the balance correction method \cite{Scharff12}.
 

\subsection{Errors in balances while exchanging  factors of conserved quantities}\label{factors}
If an exchanged quantity influences a conserved quantity, the balance of that quantity is disturbed by the approximation error of an input as described above, even if the input quantity is nonconserved. Furthermore, as the conserved quantity depends on other factors, its imbalance may persist even if the balance of the exchanged quantity is reestablished, e.g. by errors compensating each other.

The prominent example of a globally conserved quantity that is exchanged between subsystems  via its factors is energy. Think of an exchange of mechanical energy which in a cosimulation context is exchanged by passing displacement $s$ from $S_1$ to $S_2$ and the force $f$ vice versa. There is a power $p=vf$ acting at the interface between the two systems. Approximation errors on the receiving side cause an error in this power:  Into the one subsystem flows an amount of energy different from the amount of energy leaving the other. \\
Of course, the received data influences the other factors of the energy, thus one cannot directly conclude from the input error to the energy error. It is correct, anyhow, that 
this influence is not a healing one in general.

\subsection{Example for augmentation of energy}\label{exampleAugmentationOfEnergy}

Oscillating systems are most suggestive to illustrate the nature of those energy-based balance errors as they accumulate energy rapidly and furthermore the effect of extrapolation errors is comparableto that of phase angles in electric engineering. Consider the simple linear mass-spring system \eqref{springMass} given by $\dot{ \vec{ z}}
 = \ten A\vec{z}$,
 \begin{wrapfigure}{l}{0.2\textwidth}
\includegraphics{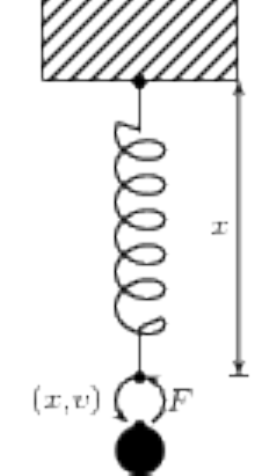}
\end{wrapfigure}
\begin{equation} 
 \ten A
= \begin{pmatrix}
0 & 1\\
-\frac{c}{m} & 0
 \end{pmatrix}. 
 \end{equation}
The solution of an initial value problem with this system is 
\begin{equation}
\vec{z} = \vec{z}_0 \operatorname{exp}\left(  \ten A(t-t_0) \right)
\end{equation}
or, equivalently, 
\begin{equation}
z_{1} = x = \sqrt{\frac{m}{c}}\dot{x}_0 \operatorname{sin}\left(  \sqrt{\frac{c}{m}}(t-t_0)\right) + x_0\operatorname{cos}\left(  \sqrt{\frac{c}{m}}(t-t_0)\right).
\end{equation}
Thus the derivatives $\dot x$ and $\ddot x$ are sines and cosines. Choosing for simplicity $v_0=0 $ and $t_0=0$ we have
\begin{equation}
\begin{split}
 x =  x_0\operatorname{cos}\left(  \sqrt{\frac{c}{m}}t\right),\qquad  \dot x = -  \sqrt{\frac{c}{m}} x_0\operatorname{sin}\left(  \sqrt{\frac{c}{m}}t\right) ,\qquad \\ \text{and } \qquad \ddot x = -  \frac{c}{m} x_0\operatorname{cos}\left(  \sqrt{\frac{c}{m}}t\right).
 \end{split}
\end{equation}
The mechanical force is $f = m\ddot x$. \\
This system was treated with a cosimulation scheme given in  Table \ref{tab:cosimSchemeBC}. 
Output of the spring is the force $f = -cx$, 
that of the mass is the velocity $v = \dot x$.\\
 The power acting on the subsystem mass is $P=F v = m\ddot x \dot x$. The energy received by the mass thus is $W = \int_{t_0}^{t_1} m\ddot x \dot x\, dt$, and as $ P = \text{const}\cdot \operatorname{sin}\left( 2  \sqrt{c/m}t\right)$   energy is of the form $ \text{const}\cdot \operatorname{cos}\left( 2  \sqrt{c/m}t\right)$. It  vanishes at $t_1$ with  $t_1-t_0=k\pi$, $k\in \mathbb N$, elsewhere it is bounded for all $t$.\\

If the split system is calculated with a cosimulation method and piecewise constant extrapolation of inputs, the force as it is seen by the mass is effectively shifted to the right. The analogy with the reactive power and the real power of an electrical network is apparent. As $\sin(\omega t)\cos(\omega t + \phi) = (\sin(2\omega t + \phi) + sin(-\phi))/2$,  work from oscillating systems with phase shift contains the integral over a constant and thus grows unbounded. The plots in figure \ref{fig:constCosimConvergence} clearly demonstrate this.\\
This kind of error occurs whenever a quantity inducing an energy on both sides of a subsystem boundary is exchanged across that boundary, making an approximation error, and thus energy appears on the one side without vanishing on the other side. 
Thus, it is also observed using degree 1 and higher extrapolation polynomials. For example,
assuming linear Hermitean extrapolation, also derivatives of these quantities are exchanged. Let the displacement of the mass increase and the neighbouring system give a force against this movement as a response, i.e. let the mass move towards a spring and compress it. The force will decelerate the mass and reverse its movement. Now it is clear that due to the piecewise linear extrapolation of the data in this zone where it is convex (concave)  a bigger displacement is assumed by the force delivering subsystem, 
making it respond with a higher force.  The energy of the system is augmented - see the plots in figure \ref{fig:linCosimConvergence}.\\

Balance correction techniques applied to the impulse as the integral of the force do not prevent  the system from picking up energy and behaving unstable, as the a posteriori refeed of force then acts at another system state than it should, as the states have changed meanwhile - here the mass has changed its velocity.\\

Mathematically, these balance errors are errors in arguments of a functional which would be conserved in the exact solution but is not in the cosimulation solution. This motivates
our method that is functional conserving, which is presented in section \ref{enforcingEnergyBalance} 
and which we finally examine for stability.

%% file: stabilityOfPowerBalancedSchemes.tex
\section{Stability of power balanced  schemes}
As discussed in Section \ref{stability} and shown in \cite{Moshagen2017} , stability for linear systems of a partly explicite scheme is not given. This section shall relate energy conservation of the method to stability. We give an outline of this section: 
\begin{itemize}
\item Switch to gradient flow view
\item  Introduce split system
    \begin{itemize}
      \item identify coupling contributions
     \item  Characerize potential conservation/dissipation properties
    \end{itemize}
\item  See method as decoupling  ODE 
      -- Insert calculation of inputs from power into orig. equations 
\item   relate  decoupled ODEs stability properties to stability of original systems 
   \begin{itemize}
    \item  show that negotiated exchange conserves $\dot{\mathcal P} \le 0$.
    \item  use Lyapunovs direct method 
    \item  additionaly, one can argue that maximum stable stepwidth for dissipative systems is augmented (method is closer to B-stable than Extrapolation of Inputs).    
    \end{itemize}
\item  if such stable subsystems ODEs are solved with methods preserving that stability, overall solution will be stable. 
\end{itemize}

\subsection{Structure of coupled system}

Many system's behavior is governed by conservation of some energy or by entropy related potential minimization. This behavior is described by relating states time derivative to the functional gradient of the potential, $\dot{\vec x} \sim \grad_{\vec x}\mathcal P^T$, usually linearly:
 \begin{equation}\label{gradientFlow}
\dot{\vec x} = - \ten M \grad_{\vec x}\mathcal P^T.
\end{equation}
The  properties of the \emph{mobility matrix}$ \ten M$ are, as 
\begin{equation}
\dot{\mathcal P}(x)= \scalar{\grad_{\vek x}\mathcal P( \vek  x)}{\dot{\vek x}}
= \scalar{\grad_{\vek x}\mathcal P( \vek  x)}{-\ten M\grad_{\vek x}\mathcal P( \vek x)^T},
\end{equation}
 connected to the properties of the system as follows:
\begin{enumerate}
\item for near-equilibrium potential driven processes, $\ten M $ is symmetric  due to the Onsager reciprocal relations, and has positive spectrum  due to vicinity to a stable equilibrium point, so is positive definite. Such a system moves  towards the potentials minimum, and 
\begin{equation}
\dot{\mathcal P}(x)= \scalar{\grad_{\vek x}\mathcal P( \vek  x)}{\dot{\vek x}}
= \scalar{\grad_{\vek x}\mathcal P( \vek  x)}{-\ten M\grad_{\vek x}\mathcal P( \vek x)^T}<0,
\end{equation}
accordingly $\le 0$ if $\ten M$ is pos. semi-def.. 
 See e.g. \cite{MoshDiff2011} for a detailed description and example.
\item for systems that preserve a total energy. For conservation, the systems evolution has to be perpendicular to the functional gradient, which is 
\begin{equation}
\scalar{\dot{\vec x}}{   \grad_{\vec x}\mathcal P} = \scalar{   \grad_{\vec x}\mathcal P}{-\ten M\grad_{\vek x}\mathcal P( \vek x)^T}=0,
\end{equation}
so skew or implementing a dirac structure  for conserving/Hamiltonian processes \cite{vanDerSchaft06}, \cite{vanDerSchaftMaschke03}. Examples for those are mechanical systems, for example the micromechanical force balance, see \cite[App 3]{MoshDiff2011}, or all Hamiltonian systems and those that can be seen as such, e.g. spring-mass systems as equation \eqref{eq:SpringMassHamiltonian} . 
\item Usually 
\begin{equation}
\ten M= \ten M_\text{skew}+\ten M_\text{pos.semidef},
\end{equation}
as systems have conserving as well as dissipative properties.
\end{enumerate}

Stability of methods is the conservation of the stability of the numerical solution of some IVP by the method. The class of  gradient flow problems that are in some sense stable is relevant: \\

\begin{definition}[stabilities]\label{def:LyapunovStability}
Let $\vek x^*$ be an equilibrium point of the ODE $\dot{ \vek{ x}} = f(\vek{ x})$ and $\phi^t\vek x$ the solution for the IVP with $\vek x(t_0)=\vek x$. Then $\vek x^*$ is
\begin{itemize}
\item \emph{stable} if $\forall \epsilon$ $\exists \delta > 0$ $\norm{\vek x-\vek x^*}<\delta$ $\Rightarrow \quad \norm{\phi^t \vek x - \vek x^* }<\epsilon$ $\forall t$
\item \emph{asymptotically stable} if $\exists r: \norm{\vek x-\vek x^*}<r$ $\Rightarrow \quad \lim_{t\longrightarrow\infty}\phi^t \vek x=\vek x^*$.
\end{itemize}
\end{definition}

Stability of an IVP can be proven using
\begin{theorem}[Lyapunovs direct method]\label{Th:LyapunovStability}
Let $\vek x^*$ be an eq. point of ODE.\\
Let $\mathcal P: V\longrightarrow \mathbb R^+:$ $\mathcal P(\vek x^*)=0$, $\mathcal P(\vek x)>0$
 $\forall \vek x\neq \vek x^*$ 
such that
\begin{equation}
\dot{ \mathcal P}(\Phi^t \vek x) = \scalar{\grad \mathcal P}{\dot{\Phi^t \vek x}}\le 0.
\end{equation}
 Then $\vek x^*$ is a stable equilibrium point as defined in \ref{def:LyapunovStability}. If $\dot{ \mathcal P}<0$, $x^*$ is an asymptotically stable eq. point.
\end{theorem}
which is proven in numerous higher analysis textbooks. It immediately follows by  $\dot{\mathcal P}(x)\le 0$ that gradient flow problems of potentials that are convex around a minimum and $\ten M$ positive definite are stable in the sense of Lyapunov.
Vice versa, it can be stated that near a stable equilibrium point,  any ODE can be approximated by a gradient flow problem. Of course, all linear ODE can be seen as gradient flow problem with respect to a quadratic functional.\\
So gradient flow problems are not only widespread, but Lyapunov stable 
 systems can be approximated by gradient flow problems. For all this,  this problem class is useful for examining  stabilities of the power balancing method.


\subsubsection{Properties of Interaction in gradient flow problem seen as subsystems}

Into the potential and its gradient flow description of the system $S$ the notion of subsystems is introduced:
With $\vek x(t) \in \mathbb R^N\times \mathbb R$ being the states of $S$,
let $I_i\subset \left\{1, 2,...N\right\} $ for $i=1...k< N$ such that $I_i \cap I_j=\emptyset$ for $i\neq j$ and $\bigcup_i I_i=\left\{1, 2,...N\right\}$.\\ 
We write
\begin{itemize}
\item $(\vek x)_{I_i}$ for vector of components (indexing operator), 
\item $\vek x_{I_i}$ for vector that contains all components of $\vek x$ that are in $I_i$ and 0 elsewhere.
\end{itemize}
The subsystems $S_i$ of $S$ now are given by the $\vek x_{I_i}$ being their states.
Without loss of generality, all $I_i$ consist of subsequent numbers. Further the gradient flow context implies that what was modelled as input before e.g. in \eqref{} depends on states such that it can be expressed by  the states $\vek x_{I_j}$ and applying the replacement $\vek u_{ij}=\vek x_{I_j}$ is no further loss of generality.\\

With this, we can identify subsystems energy gain as
\begin{align}
\dot{\mathcal P_i}(x):= P_i
&= \scalar{\grad_{\vek x_{I_i}}\mathcal P( \vek  x)}{-\ten M\grad_{\vek x}\mathcal P( \vek x)^T}
\\&= \scalar{(\grad_{\vek x}\mathcal P( \vek  x))_{I_i,.}}{-(\ten M)_{I_i,.}\grad_{\vek x}\mathcal P( \vek x)^T}.
\end{align}
Using this notation, to identify the  interaction between subsystems it is split into $I_i$-block lines
\begin{align}
\dot{\mathcal P}(x)=\sum_iP_i
&= \sum_i\scalar{\grad_{\vek x_{I_i}}\mathcal P( \vek  x)}{-\ten M_{I_i,.}\grad_{\vek x}\mathcal P( \vek x)^T}
\\ \nonumber  &= \sum_i\Big(\scalar{\grad_{\vek x_{I_i}}\mathcal P( \vek  x)}{-\ten M_{I_i,I_i}\grad_{\vek x}\mathcal P( \vek x)^T}
\\&\qquad+\sum_{j\neq i}\scalar{\grad_{\vek x_{I_i}}\mathcal P( \vek  x)}{-\ten M_{I_i,I_j}\grad_{\vek x_{I_j}}\mathcal P( \vek x)^T}
\Big).\label{eq:PotProduction}
\end{align}
and then
splitting the block lines into blocks, mentioning only $k$ and $l$-containing expressions, 
\begin{align}
&\dot{\mathcal P}(x)=P_k+P_l+ ...
\\&= 
\scalar{\begin{pmatrix}. 
\\ (\grad_{\vek x}\mathcal P( \vek x))_{I_k}
\\.
\\ (\grad_{\vek x}\mathcal P( \vek x))_{I_l} 
\\.
\end{pmatrix}}
{
\begin{pmatrix}
& &*\\
... &-(\ten M)_{I_k,I_k}
& ...
& -(\ten M)_{I_k,I_l} &...
\\ & &*
\\
... &-(\ten M)_{I_l,I_k}
& ...
&-(\ten M)_{I_l,I_l} &...
\\ & &*
\end{pmatrix}
\begin{pmatrix}
. 
\\ (\grad_{\vek x}\mathcal P( \vek x))_{I_k}
\\.
\\ (\grad_{\vek x}\mathcal P( \vek x))_{I_l}
\\.
\end{pmatrix}}
\\&=\underbrace{\scalar{\grad_{\vek x_{I_k}}\mathcal P( \vek  x)}{-\ten M_{I_k,I_k}\grad_{\vek x_{I_k}}\mathcal P( \vek x)^T}}_{P_{kk}}
+\underbrace{\scalar{\grad_{\vek x_{I_k}}\mathcal P( \vek  x)}{-\ten M_{I_k,I_l}\grad_{\vek x_{I_l}}\mathcal P( \vek x)^T}}_{P_{kl}}\nonumber
\\&+\underbrace{\scalar{\grad_{\vek x_{I_l}}\mathcal P( \vek  x)}{-\ten M_{I_l,I_l}\grad_{\vek x_{I_l}}\mathcal P( \vek x)^T}}_{P_{ll}}
+\underbrace{\scalar{\grad_{\vek x_{I_l}}\mathcal P( \vek  x)}{-\ten M_{I_l,I_k}\grad_{\vek x_{I_k}}\mathcal P( \vek x)^T}}_{P_{lk}},\nonumber
\\ &\qquad + ...
\end{align}
we identify 
\begin{equation}\label{eq:P_kl}
P_{kl}:=\scalar{(\grad_{\vek x}\mathcal P( \vek x))_{I_k}}{-(\ten M)_{I_k,I_l}(\grad_{\vek x}\mathcal P( \vek x)^T)_{I_l}}
\end{equation}
as the \emph{potential production}  in $S_k$ by $S_l$s variables, 
or power acting from subsystem $l$ onto subsystem $k$ and 
\begin{equation}\label{eq:P_kk}
P_{kk}:=\scalar{(\grad_{\vek x}\mathcal P( \vek x))_{I_k}}{-(\ten M)_{I_k,I_k}(\grad_{\vek x}\mathcal P( \vek x)^T)_{I_k}}
\end{equation}
as $S_k$s internal potential change.

\subsubsection{Properties of Potential Production contributions}
Let, as discussed above, $\ten M=\ten M_\text{symm}+\ten M_\text{skew}$. Due to \eqref{eq:P_kk}, $P_{kk}$ is
\begin{itemize}
\item energy absorbing iff $\rho(\ten M_{\text{symm},kk})\subset \mathbb R^+$ and $\exists \lambda_i> 0$ 
\item energy conserving iff $\ten M_{kk}=\ten M_{\text{skew},kk}$: It is then a Hamiltonian system with inputs. 
\end{itemize}
If $\ten M$ is symmetric, then 
\begin{equation}\label{eq:PotProdSym}
\begin{aligned}
P_{kl}&= -\scalar{(\grad_{\vek x}\mathcal P( \vek  x))_{I_k}}{(\ten M)_{I_k,I_l}(\grad_{\vek x}\mathcal P( \vek x)^T)_{I_l}}
\\&= -\scalar{(\ten M)_{I_k,I_l}^T(\grad_{\vek x}\mathcal P( \vek  x))^T_{I_k}}{(\grad_{\vek x}\mathcal P( \vek x))_{I_l}}
\\&=-\scalar{(\ten M)_{I_k,I_l}(\grad_{\vek x}\mathcal P( \vek  x))^T_{I_k}}{(\grad_{\vek x}\mathcal P( \vek x))_{I_l}}\qquad \text{ (symmetry)}
\\&=-\scalar{(\grad_{\vek x}\mathcal P( \vek x))_{I_l}}{(\ten M)_{I_k,I_l}(\grad_{\vek x}\mathcal P( \vek  x))^T_{I_k}}=P_{lk}.
\end{aligned}
\end{equation}
which can be expressed in words as: The states adjacent to the $S_l$--$S_k$ boundary influence Potential on both subsystems equally. For example, when coupling two spatial domains with a PDE-governed field on them, this means that DOFs that  neighbour the other domain evolve such that the potential production (e.g. entropy production) is equal for both domains.
\\If $\ten M$ skew, then analogously 
\begin{equation}\label{eq:PotProdSkew}
P_{kl}=- P_{lk}
\end{equation}
This implements a real potential flow across the boundary: $S_k$ takes what $S_l$ loses. Those flows cancel out in the overall potential production \eqref{eq:PotProduction}.
\\
If the mobility matrix has both nonzero $\ten M_\text{symm}$ and $\ten M_\text{skew}$ contribution, the potential flow has real flow contributions that are no productions and balanced production contributions contributions .

\begin{remark}[Relation to Port Hamiltonian systems]
The setting in which $S$ is given above makes it a \emph{Port-Hamiltonian System} according to \cite{vanDerSchaft06}, which is a system of the shape
\begin{equation}\label{eq:PortHamiltonian}
\dot{\vec x}= \left(\ten J - \ten R\right)\grad_{\vek x}\mathcal P( \vek x) + \ten G\vek u,
\end{equation}
with skew $\ten J $ that corresponds to $\ten M_\text{skew}$, and $\ten R$ is symmetric positive definite and corresponds to $-\ten M_\text{symm}$.
 If a $P_{kk}$ is
\begin{itemize}
\item energy absorbing, i. e. $\ten M_{kk}$  is symmetric postive definite, then it is a dissipative port
\item energy conserving, i.e.  $\ten M$ skew, it is a Port-Hamiltonian subsystem. 
\end{itemize}
\end{remark}

\subsubsection{Structure of decoupled and power balanced system}

We now derive the ODE induced by cosimulation scheme from the original ODE by replacing $S_i$s inputs $\vek x_{I_j}$, $j \neq i$, by the explicit expression replacing them:
From directly inserting $\operatorname{Ext}(\vek x_{I_j})$ into $\dot{\vec x} = - \ten M \grad_{\vec x}\mathcal P$ 
\begin{equation}\label{eq:ExtrapInducedODE}
\dot{ \vek x}_\text{Extrap} = \sum_i \ten M_{I_i,.}\left(\grad_{\vek x}\mathcal P(\vek x_{I_i}, \operatorname{Ext}((\vec x_{I_j})_{j\neq i})\right)
\end{equation}
results, which is no gradient flow any more and in general will be instable ( see section \ref{stability}).\\
Analogously, power balanced method induces an ODE:
\begin{equation}\label{eq:BalanceInducedODE}
\dot{ \vek x}_\text{bal} = \sum_i\dot{ \vek x}_{I_i}=\sum_i \ten M_{I_i,.}  \grad_{\vek x}\mathcal P(\vek x_{I_i}, (\operatorname{Ext}(\vec x_{I_j\setminus j}), P_{ij}^{-(\vec x)_j}(\operatorname{Ext}\hat P_{ij}))_{j\neq i})
\end{equation}
The index $j$ is, without loss of generality, used to denote the component of $\vec x_{I_j}$ that was replaced by $P_{ij}$. The argument $(\operatorname{Ext}(\vec x_{I_j\setminus j}), P_{ij}^{-(\vec x)_j}(\operatorname{Ext}\hat P_{ij}))_{j\neq i}$ is a vector, as in fact as many $j \neq i$  as couplings appear in the argument. For readability, $\operatorname{Ext}(\vec x_{I_j\setminus j})$ will be omitted from now on.

 \subsubsection{Stability inheritance}
 
Before it is sketched how from properties of the method induced ODE the stability properties of our scheme is derived, remember the stability concepts:\\
First, from  the dissipativity of an equations right hand side
\begin{equation}\label{eq:dissipativeRHS}
\scalar{\vek f (\vek x)- \vek f(\overline{\vek x}) }{\vek x-\overline{\vek x} }\le 0
\end{equation}
 follows the \emph{non-expansivity} of its evolution $\vec \Phi^\tau \vec x$
\begin{equation}
\norm{\vek \Phi^\tau\vek x- \vek \Phi^\tau\overline{\vek x} }\le \norm{\vek x-\overline{\vek x} }\qquad \text{for all } \vek x,\,\overline{\vek x}
\end{equation}
(see \cite[th. 6.49 ]{DeuflhardBornemann94}), which if inherited to numerical solution is the \emph{B-stability} of the method (\cite[Section 6.3]{DeuflhardBornemann94}). \\
Second,
remember A-Stability and its vector valued generalization, the linear stability as discussed in Section \ref{stability}. Remember further Definition of Lyapunov stability 
 \ref{def:LyapunovStability}.

We now use the following obvious arguing:  Let original ODE
\begin{equation}
\begin{cases}
\text{have a stable point at } x^*
\\ \text{be linear and stable}
\\ \text{be dissipative.}
\end{cases}
\end{equation}
Then, if ODE induced by method inherits this property, which means it still has a stable point at $x^*$/ is still linear and stable/dissipative, 
all subproblems $\dot{\vek x}_{I_i}$ have that property.\\
Then all subproblems numerical solution $\Psi{\vek x}_{I_i}$ is stable  if solved with a
\begin{equation}
\begin{cases}
\text{ stable for that } \vek x^*_{I_i}
\\ \text{A-stable}
\\ \text{B-stable}
\end{cases}
\text{ method.}
\end{equation}
Then the 
overall solution is stable and so the split method is
\begin{equation}
\begin{cases}
\text{stable at }\vek  x^*
\\ \text{A-stable}
\\ \text{B-stable}
\end{cases}
\end{equation}
if the methods applied on the subsystems are.
Stability of power balanced scheme now is shown using 
Lyapunovs direct method, Theorem \ref{Th:LyapunovStability}.\\
By applying this method one gains the stability result for the power balanced method:
\begin{theorem}
For a Lyapunov stable (asymptotically stable) initial value problem (IVP), the IVP resulting from the energy balancing method is also stable (asymptotically stable). 
\end{theorem}
\begin{proof}
Inserting \eqref{eq:BalanceInducedODE} for $\dot{\vek x}$,
\begin{equation}
\begin{split}
\dot{\mathcal P}(\Phi^t_\text{bal}\vek x)=\scalar{\grad_{\vek x}\mathcal P( \Phi^t\vek x)}{\dot{ \Phi^t\vek x)}}
\\=\scalar{\grad_{\vek x}\mathcal P( \Phi^t\vek x)}{\sum_i \ten M_{I_i,.}  \grad_{\vek x}\mathcal P( \Phi^t \vek x_{I_i}, (P^{-(\vec x)_j}(\operatorname{Ext}\hat P_{i,j}))_{j\neq i})}  
\\=\sum_i\scalar{\left(\grad_{\vek x}\mathcal P( \Phi^t\vek x)\right)_{I_i}}{ (\ten M)_{I_i,I_i}  \left(\grad_{\vek x}\mathcal P( \Phi^t \vek x_{I_i}, (P^{-(\vec x)_j}(\operatorname{Ext}\hat P_{i,j}))_{j\neq i})\right)_{I_i}}
\\+\sum_i\sum_{j\neq i}\scalar{\left(\grad_{\vek x}\mathcal P( \Phi^t\vek x)\right)_{I_i}}{ (\ten M)_{I_i,I_j}  \left(\grad_{\vek x}\mathcal P( \Phi^t \vek x_{I_i}, (P^{-(\vec x)_j}(\operatorname{Ext}\hat P_{i,j}))_{j\neq i})\right)_{I_j}}
\\
=\sum_iP_{ii}+\sum_i\sum_{j\neq i}\hat P_{ij}
\end{split}
\end{equation}
results, the  last equality by identifying subsystem internal production and exchange:
Potential flow, as always, is:  
$$
P_{\text{bal},ij}=\scalar{\left(\grad_{\vek x}\mathcal P( \Phi^t\vek x)\right)_{I_i}}{ (\ten M)_{I_i,I_j}  \left(\grad_{\vek x}\mathcal P( \Phi^t \vek x_{I_i}, (P^{-(\vec x)_j}(\operatorname{Ext}\hat P_{i,j}))_{j\neq i})\right)_{I_j}}
$$
and by construction of the power negotiating method, without loss of generality $P_{\text{bal},ij}=\hat P_{ij}$
 and $P_{\text{bal},ji}= -\hat P_{ij}$ - the $x_j=P^{-(\vec x)_j}(\operatorname{Ext}\hat P_{i,j})$ are determined with respect to that. So the second sum cancels out,
\begin{equation}\label{eq:dotPbal}
\dot{\mathcal P}(\Phi^t_\text{bal}\vek x)=\sum_iP_{ii}\le 0,
\end{equation}
 as $P_{ii}\le0$ for all $i$. So theorem \ref{Th:LyapunovStability} can be applied.
\end{proof}
\paragraph{Remark} The cancellation $\sum_i\sum_{j\neq i}\hat P_{ij}=0$ 
due to $\hat P_{ji}= -\hat P_{ij}$  obviously is consistent to the cancellation of real flows due to $\ten M_{skew}$ in the   original model - its flows also balance according to \eqref{eq:PotProdSkew}: $ P_{ji}= - P_{ij}$.\\
Is it consistent with original model for the flows due to the symmetric part of $\ten M$ like in \eqref{eq:PotProdSym}?\\
Symmetric part is  the potential production on subsystem boundary DOFs, flowing equally into both subsystems. In the original ODE, $ P_{ji}=  P_{ij}$, and  from negative semidefiniteness of $\dot{\mathcal P}$,  the restriction
\begin{equation}
P_{ii}+P_{jj} \le -P_{ij}-P_{ji}. 
\end{equation}
holds. In power balanced scheme, it is replaced by
\begin{equation}
P_{\text{bal},ij}+P_{\text{bal},ji}=2P_{\text{bal},ij}=\hat P_{ji}+\hat P_{ij}=0.
\end{equation}
This means: Using power balanced decoupling scheme here is a  change in the model -- a bit of damping is lost --
 leading to a small but $O(1)-$\emph{Modeling error  in this case!}
\input{towardsBStability.tex}

%% file: towardsBStability.tex
\subsection{Towards B-stability - inheritance of dissipativity to the method induced ODE?}

For a system that is governed by \eqref{gradientFlow} for convex $\mathcal P$ one can evaluate  the dissipativity relation \eqref{eq:dissipativeRHS} as follows
\begin{multline}\label{dissipativityOfGrad}
 \scalar{\vec f (\vec x)-\vec f(\overline{\vec x})}{\vec x - \overline{\vec x}}
 =\scalar{-\ten M\left(\grad_{\vec x}\mathcal P (\vec x)-\grad_{\vec x}\mathcal P (\overline{\vec x})\right)}{\vec x - \overline{\vec x}}\\
= -\scalar{\int\limits_0^1\ten M\left( D^2_{\vec x}\mathcal P (\vec x + \theta(\overline{\vec x}-\vec x))\right)(\vec x - \overline{\vec x})\,d\theta}{\vec x - \overline{\vec x}}
\\= -\int\limits_0^1 \scalar{\ten M\left( D^2_{\vec x}\mathcal P (\vec x + \theta(\overline{\vec x}-\vec x))\right)(\vec x - \overline{\vec x})}{\vec x - \overline{\vec x}}\,d\theta.
\end{multline}
This is lower than 0 for all $\vek x$ and  $\overline{\vek x}$ if and only if $\ten M D^2_{\vec x}\mathcal P$ is positive definite for all  $\vek x$, lower or equal 0 for positive semidefiniteness.
 The Hesse matrix $D^2_{\vec x}\mathcal P$  is positive definite as of a convex functional. 
The result is stated formally:\\
\begin{lemma}\label{dissipativityOfGradTh}
Given a system of ODEs on whose states $\vek x$ an energy  functional $\mathcal P(\vek x)$ is defined, 
let the ODEs be defined by its gradient flow $\grad_{\vec x} \mathcal P $  by   
\begin{align}
\dot{\vec x} &= - \ten M \grad_{\vec x}\mathcal P\\
\vec x(t_0) &=\vec x_0.
\end{align}
If the right hand side $- \ten M \grad_{\vec x}\mathcal P$ is dissipative, i.e. 
\begin{equation}
\scalar{-\ten M\grad_{\vec x}\mathcal P (\vec x)-(-\ten M)\grad_{\vec x}\mathcal P (\overline{\vec x})}{\vec x - \overline{\vec x}}\le 0,
\end{equation}
then $\ten M\ten D^2 \mathcal P$ is positive semi-definite.
Vice versa, if $\ten M\ten D^2 \mathcal P$ is positive definite, the right hand side $- \ten M \grad_{\vec x}\mathcal P$ is dissipative.\\
Inequality implies positive definiteness and vice versa.
\end{lemma}

 A big subclass of split gradient flow systems is given by systems $S=\cup {S_i}$ in which all couplings are due to $\ten M_\text{skew}$, and all subsystems energies are due to subsystem variables only, which is $\mathcal P_i= \mathcal P_i(\vek x_i)$. Example \eqref{eq:SpringHamiltonianEnergy} gives a hint that the class of problems for which the assumptions are valid is relevant. 
 For this class it holds that 
\begin{equation}
\dot{\mathcal P}(\vek x)=\sum_iP_{ii}=\dot{\mathcal P}_\text{bal}(\vek x).
\end{equation}
The first equality follows by $P_{kl}=- P_{lk}$ \eqref{eq:PotProdSkew} from
\begin{equation}
\dot{\mathcal P}(\vek x)=\sum_iP_{ii}+\sum_i\sum_{j\neq i} P_{ij}
\end{equation}
the second is equation \eqref{eq:dotPbal} - saying, by $\mathcal P_i= \mathcal P_i(\vek x_i)$ and chancellation of flows, the potential production of the power balanced method is independent of extrapolations. Moreover, if written as scalar product using $P_{ii} =\scalar{\grad_{\vek x_{I_i}}\mathcal P_i( \vek  x)}{-\ten M\grad_{\vek x}\mathcal P_i( \vek x_{I_i})^T}$,
\begin{equation}
\scalar{\grad_{\vek x}\mathcal P( \vek  x)}{-\ten M\grad_{\vek x}\mathcal P( \vek x)^T}=\scalar{\grad_{\vek x}\mathcal P( \vek  x)}{-\sum_i\ten M_{I_iI_i}\grad_{\vek x}\mathcal P( \vek x)^T},
\end{equation}
which says that the projection of $\dot{\vek x}$ and $\dot{\vek x}_\text{bal}$ onto the gradient of the potential is the same for the ODE and the balanced method induced ODE.\\
This projection property  is independent of extrapolation, the property is given for all exchange step lengths if given for one.\\
As before, it here becomes obvious that for $\ten M= \ten M_{symm}+ \ten M_{skew}$ only $\ten M_{symm}$ induces potential production, the skew part has no influence on it.

%% file: numResultsEnergyII.tex
\section{Comparison of numerical results and Discussion}
\label{}
\subsection{Convergence}

%

\begin{figure}
\includegraphics[ width=0.8\textwidth ]
{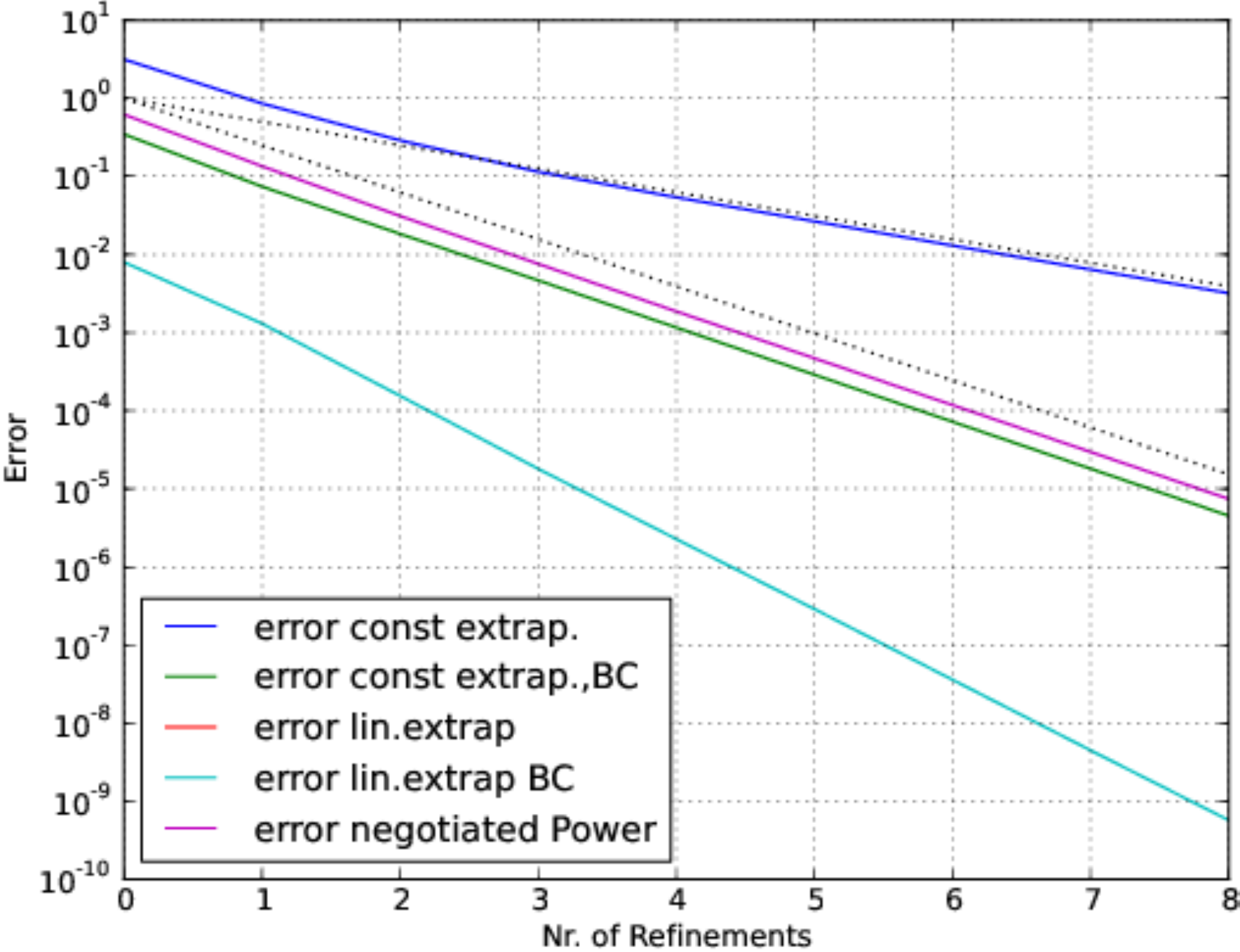}
\caption{\label{comparisonPlotConvergence} Convergence, $T_\text{end}=20$, subsystems refinement decisions left to subsystems solvers, \emph{dopri5} used on subsystems. In this setting, extrapolation error dominates.   }
\end{figure}


Simulations of the spring-mass problem \eqref{springMass} in its splitted form  as described in Table \ref{CosimSchemesSpringMass}, moreover  with  balance correction as in Table \ref{tab:cosimSchemeBC}, \cite{Scharff12}, \cite{ScharffMoshagen2017} and by negotiated power scheme (section \ref{enforcingEnergyBalance}) have been executed. Expressed as gradient flow problem, more precisely as a Hamiltonian system, as in \eqref{eq:PortHamiltonian} ,
it is written as 
\begin{equation}\label{eq:SpringMassHamiltonian}
\begin{pmatrix}
\dot{\vec q}\\ 
\dot{\vec p} 
\end{pmatrix}
=\begin{pmatrix}
0   & 1\\
-1  & 0
\end{pmatrix}
\begin{pmatrix}
\frac{\partial \mathcal H}{\partial \vec q}\\
\frac{\partial \mathcal H}{\partial \vec p}
\end{pmatrix}
=\begin{pmatrix}
0   & 1\\
-1  & 0
\end{pmatrix}
\begin{pmatrix}
cx\\
v
\end{pmatrix}.
\end{equation}
where the energy of this system is
\begin{equation}\label{eq:SpringHamiltonianEnergy}
\mathcal H = \frac{1}{2}  mv^2 +\frac{1}{2} \sum cx^2
= \frac{1}{2}  p^2/m +\frac{1}{2}  cq^2,
\end{equation}
 and so
\begin{align}
\frac{\partial \mathcal H}{\partial q} = cq \\
\frac{\partial \mathcal H}{\partial p} =\frac{1}{m}p = v. 
\end{align}
As Hamiltonian system, this system has skew mobility matrix, is energy conserving and has a stable solution.\\
for the series of exchange step sizes $H={0.2,0.1,0.05...}$. As a subsystem solver \emph{dopri5} was used, a  one-step method is necessary due to reasons that will follow. In this setting, the extrapolation error dominates $\epsilon_\Delta$.\\
Plot \ref{comparisonPlotConvergence} shows that convergence rates predicted by \eqref{convErrBoundingFunction} are met, also by the negotiated power method, which has extrapolation order 2. The error  of the  balance correction method behaves better than predicted by Theorem \eqref{th:convergenceBC}: it reduces approximately by 1/8 in each refinement step for linear extrapolation, and 1/4 for constant extrapolation. In the derivation of the estimate, the balance correction was treated as an arbitrary perturbation - in fact, it reduces the extrapolation error, and the numerical experiment suggests that this reduction leads to a gain of one in the order of convergence.
%
\begin{figure}
\includegraphics[ width=0.3\textwidth ]
{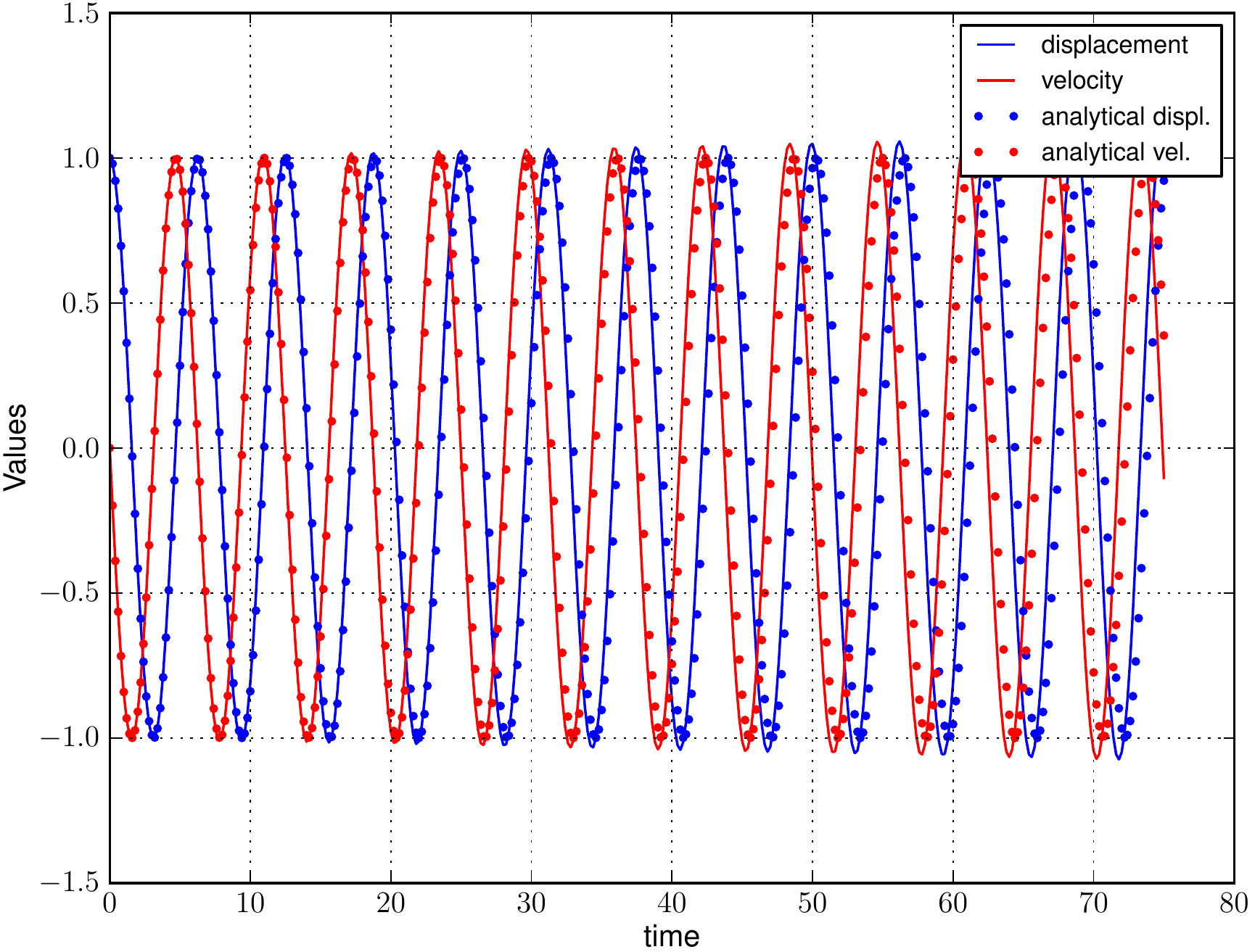}
\includegraphics[ width=0.3\textwidth ]
{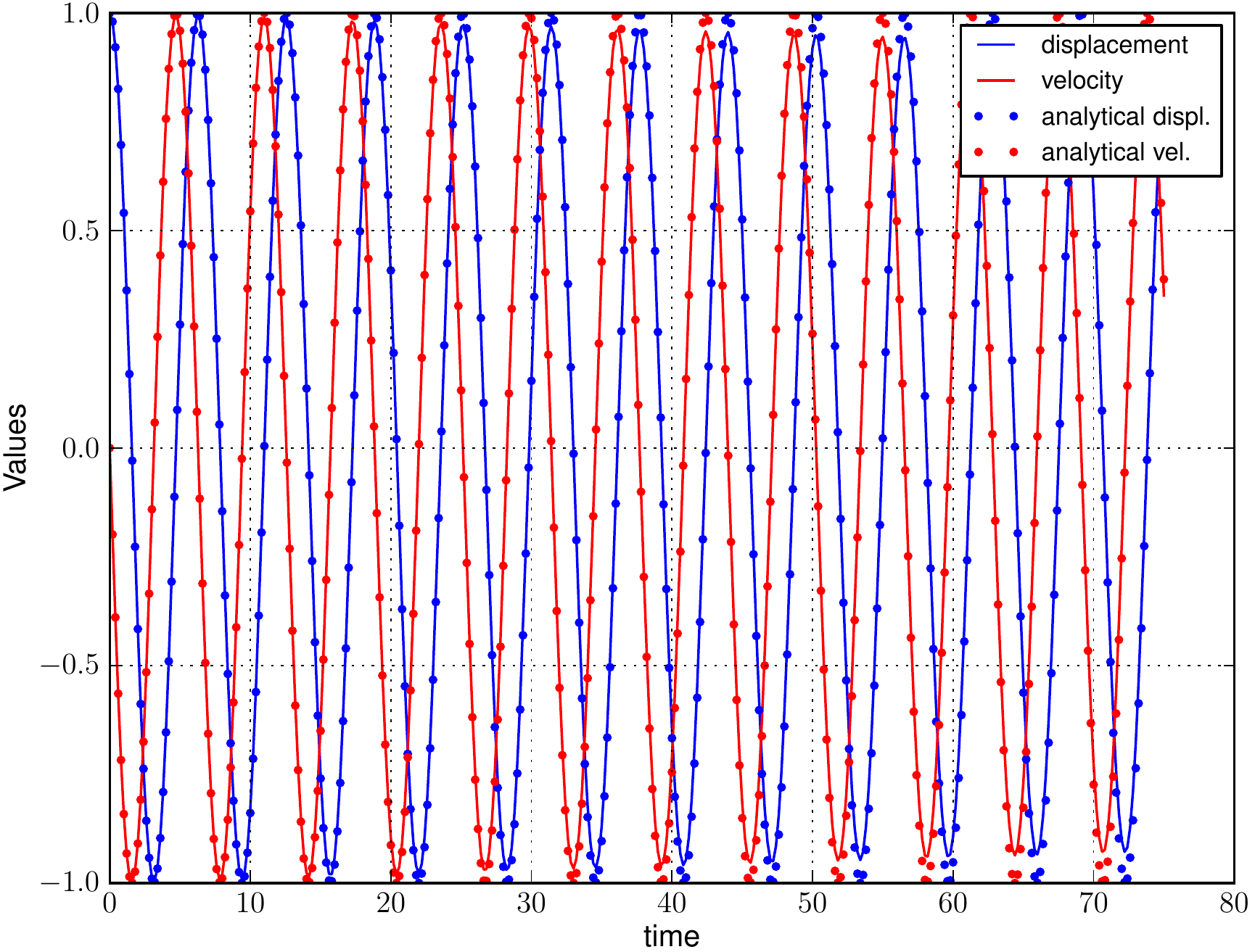}
\includegraphics[ width=0.3\textwidth ]
{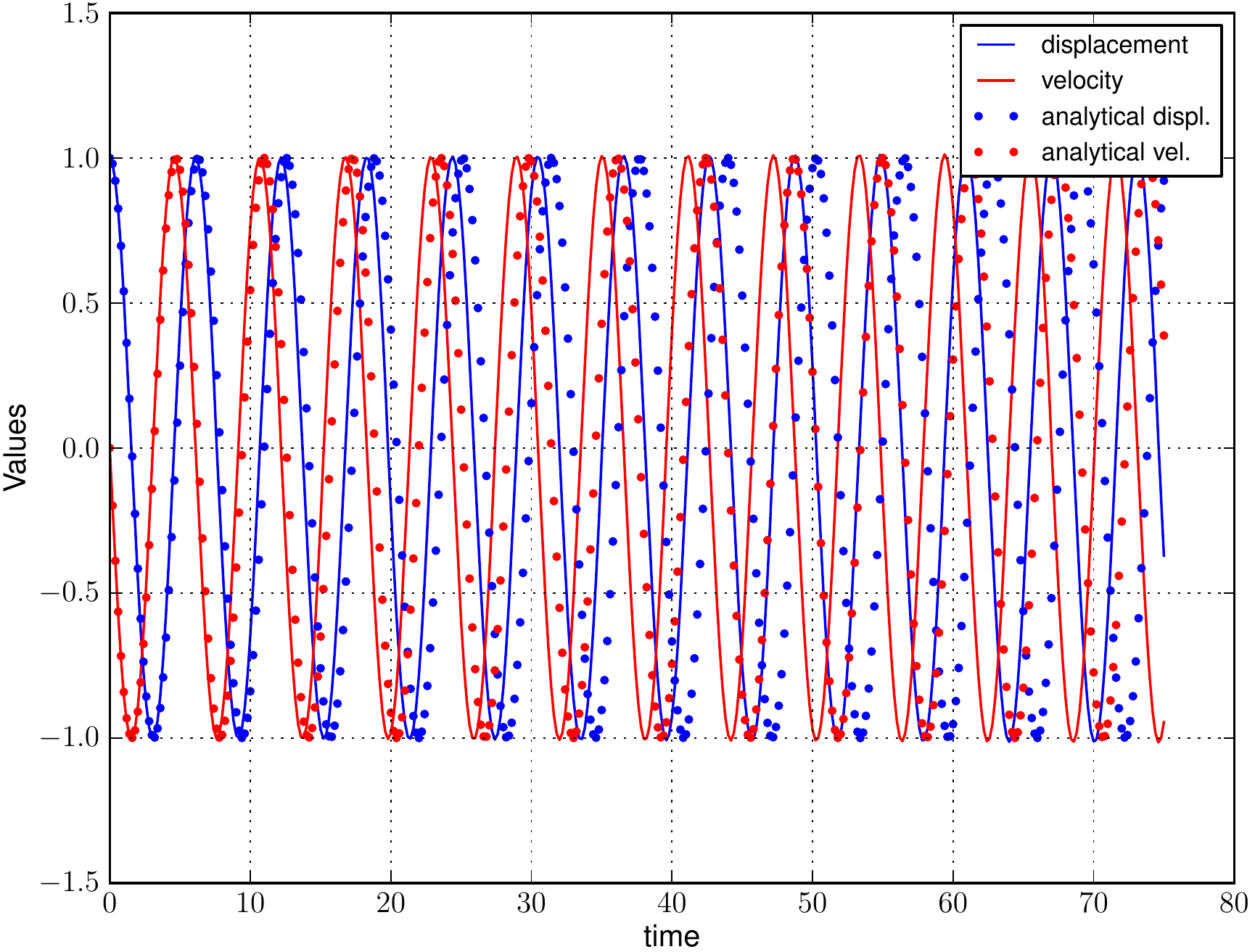}
\caption{\label{fig:stability_h20} Stability of cosimulation schemes apllied to spring-mass system: Left: Linear extrapolation, middle: Linear extrapolation with balance correction, right: Power balanced scheme. $T_\text{end}=75$, exchange stepwidth $H=0.2$, subsystems refinement decisions left to subsystems solvers, stable \emph{vode} used on subsystems.   Power balanced scheme is stable and conserves energy. }
\end{figure}
\subsection{Stability}
\begin{figure}
\includegraphics[ width=0.3\textwidth ,trim={10cm  0cm  0cm 6cm},clip]
{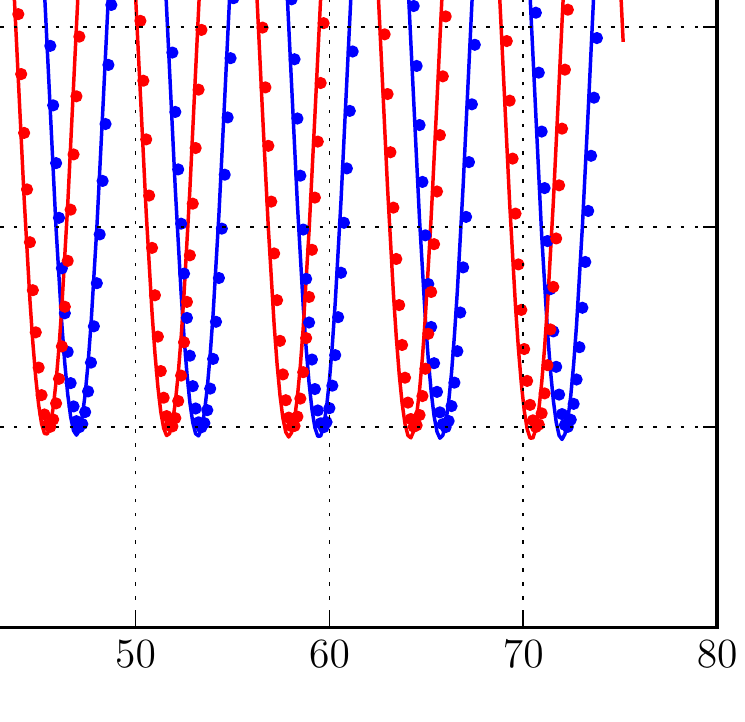}
\includegraphics[ width=0.3\textwidth ,trim={10cm  0cm  0cm 6cm},clip]
{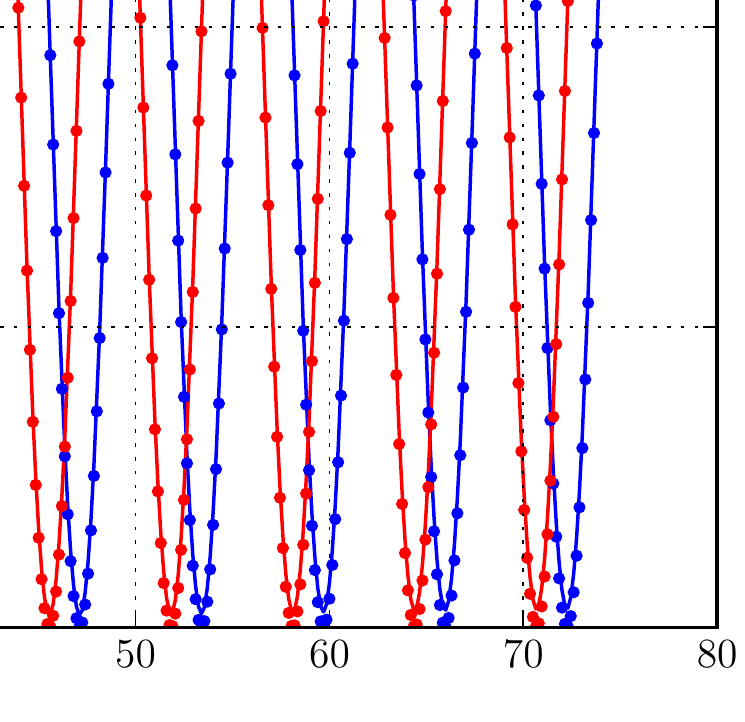}
\includegraphics[ width=0.3\textwidth ,trim={10cm  0cm  0cm 6cm},clip]
{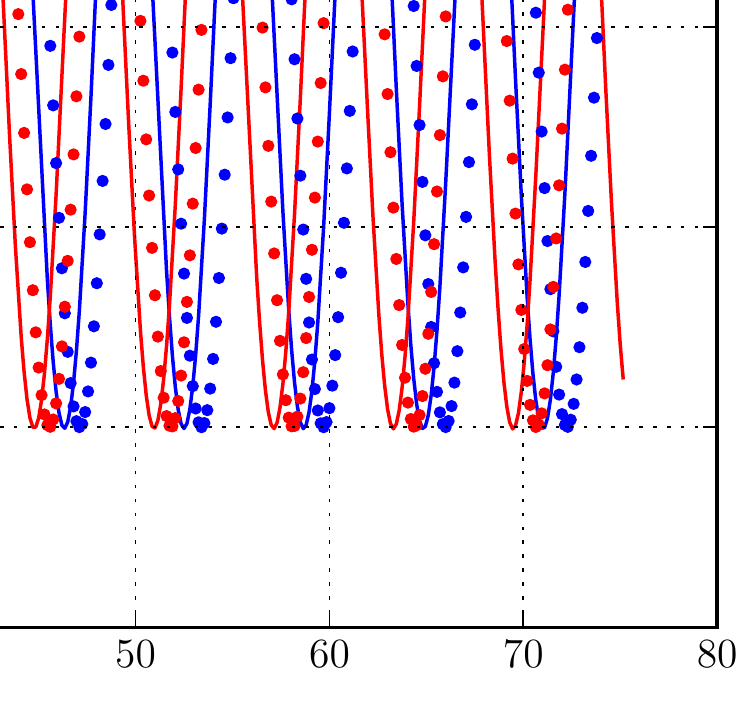}
\caption{\label{fig:stability_h15} Numerical setting as for figure \ref{fig:stability_h20}, but  exchange stepwidth $H=0.15$. Lower right corner is shown magnified to reveal amplitude.}
\end{figure}

The numerical examination reveals that stability in practice is, albeit given for bigger $H$ than for all other coupled methods, not global. It turns out that this is due to the difficulties when calculating the inverse of Power, which is treated in section \ref{inverseOfP}. There, it is also shown that without that phenomenon, the energy of the system would be conserved.\\
Figure \ref{fig:stability_h20} shows  the solution of the spring-mass system simulated using $H=0.2$ and linear extrapolation without and with balance correction and (right) the power balanced scheme. The latter one is the only one that conserves the energy. This still holds for $H=0.15$, see Figure \ref{fig:stability_h15}. For smaller $H$ all methods become stable.
\subsection{Pitfalls}
\subsubsection{Inversion of $\hat P$}\label{inverseOfP}
As mentioned in \ref{exampleNegPower}, the inverses $P_i^{-(\vec u_{ij})_j}$
have to be calculated, which e.g. means finding $(u_{\vec 12,\text{Std}})_j$ such that
\begin{equation}\label{inverseOfPower}
 \operatorname{Ext}\hat P(t) 
= P_{21}(\vec x_1, (\vec u_{12})_1, ...(\vec u_{12})_{n-1},(\vec u_{12,\text{Std}})_n)(t)
\end{equation}
 - there, e.g. $P_1^{-v}    =\frac{\operatorname{Ext}(\hat P)}{cs}$ and $P_2^{-f}  =\frac{\operatorname{Ext}(\hat P)}{v}$. For exact values, those are well defined and bounded as $P\longrightarrow 0$ if $s\longrightarrow 0$ and  if $v\longrightarrow 0$. On a computer, they might be undefined or arbitrary large.  One has to switch to d'Hopitals rule for calculation near denominators zeros.\\
 \begin{figure}[h!tb]
\includegraphics[ width=0.3\textwidth ]
{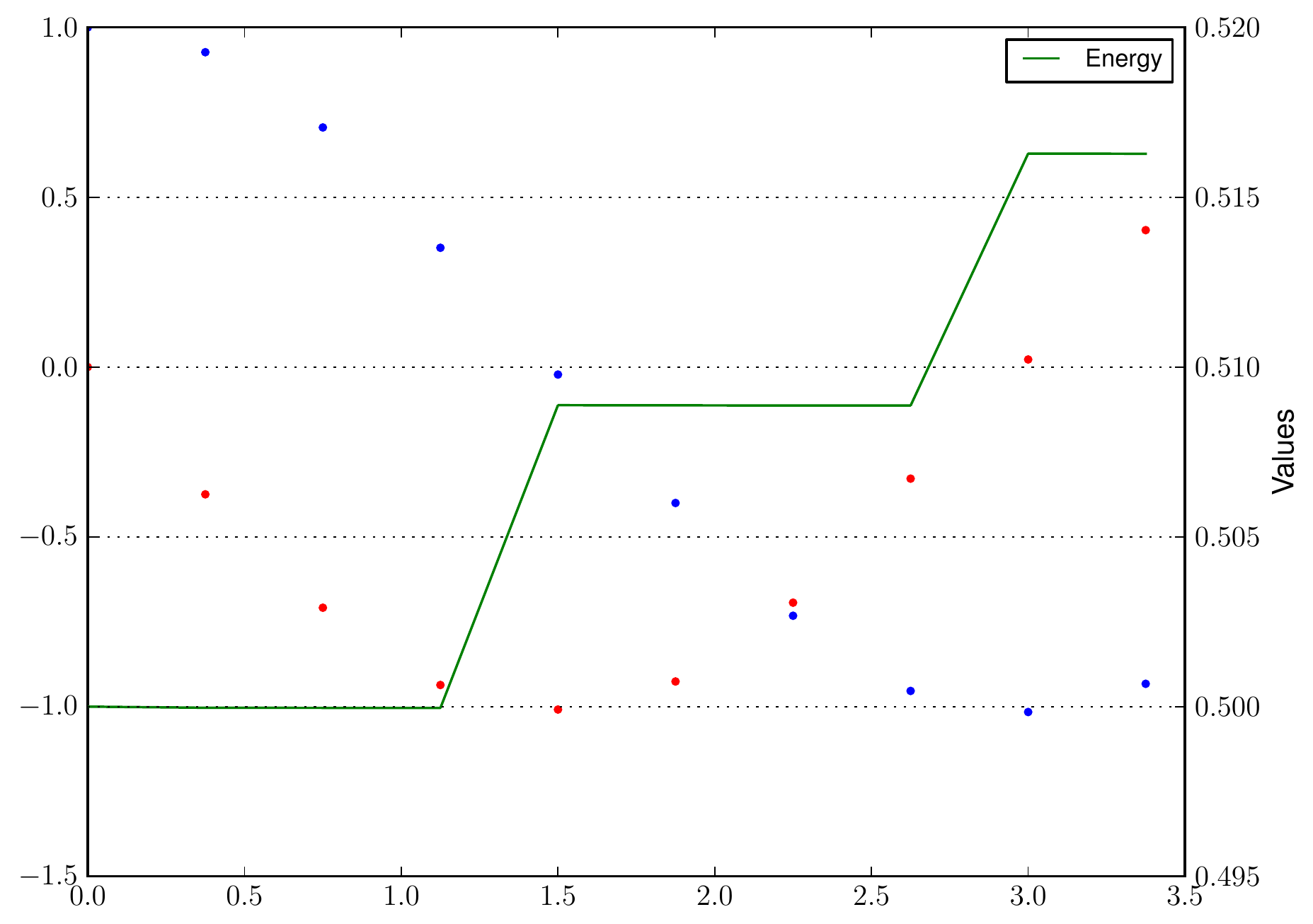}
\includegraphics[ width=0.3\textwidth ]
{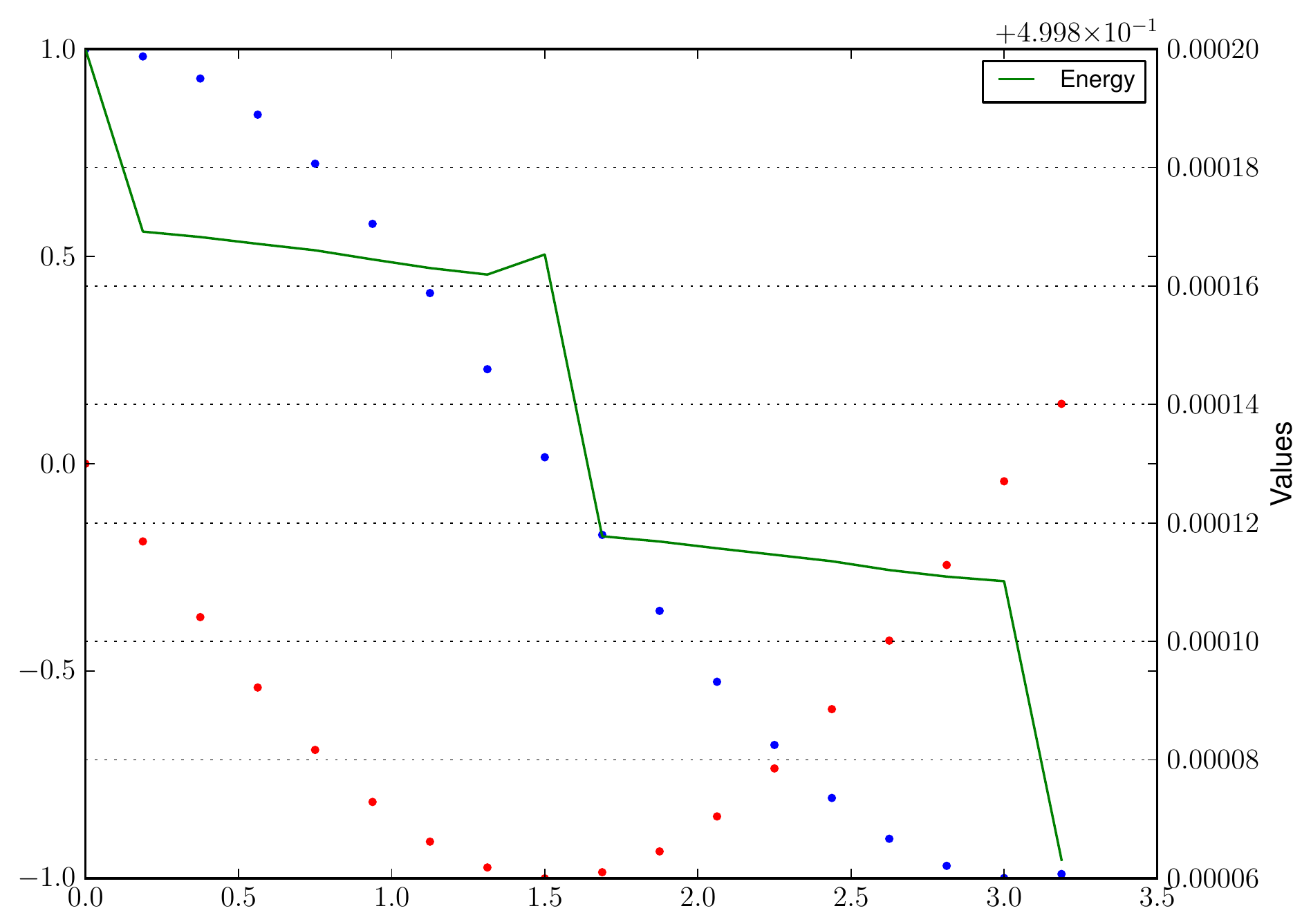}
\includegraphics[ width=0.3\textwidth ]
{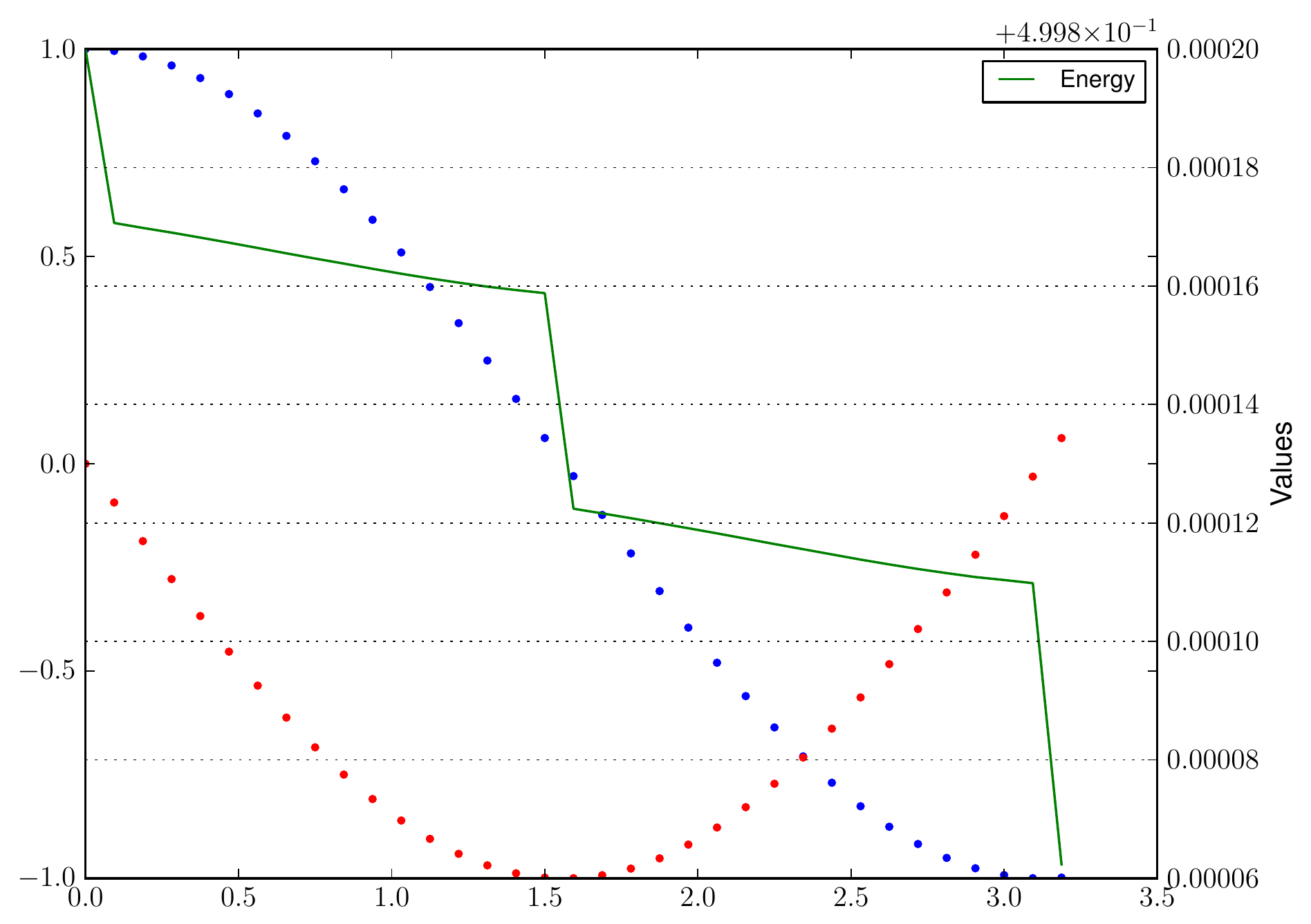}
\caption{\label{energyProductionAtZeros} Energy production when denominator of inverse of power crosses zero for $H=0.375s$,  $H=0.1825s$, $H=0.09125s$, from left to right. The plots show that zero crossings produce or consume energy, while energy is conserved away from them.}
\end{figure}
It turns out that the effect of the errors  committed during solving \eqref{inverseOfPower} near zeros of the denominator  is acting as an energy source for bigger step sizes and such is threatening stability,  even if solved with d'Hopitals rule. Away from those zeros, energy of the system is conserved, and for systems where such divisions do not appear there will be no such unphysical sources. Furthermore, even for our case the power balanced method is stable for much bigger stepsizes than the formerly discussed Cosimulation methods.

